\documentclass[10pt]{amsart}

\usepackage{amsmath, amsfonts, amsthm, amstext, amssymb,bbm,enumerate,fullpage,mathrsfs, stix2 , tikz-cd}

\usepackage{array}

\usepackage[all]{xy}

\usepackage[hidelinks,backref]{hyperref}

\usepackage[alphabetic]{amsrefs}

\newcommand{\longto}{\longrightarrow}

\theoremstyle{plain}
\newtheorem{theorem}{Theorem}[subsection]
\newtheorem{lemma}[theorem]{Lemma}
\newtheorem{proposition}[theorem]{Proposition}
\newtheorem{corollary}[theorem]{Corollary}
\theoremstyle{definition}
\newtheorem{definition}[theorem]{Definition}
\newtheorem{notation}[theorem]{Notation}

\newtheorem{example}[theorem]{Example}
\newtheorem{question}[theorem]{Question}

\newtheorem{rmk}[theorem]{Remark}
\newtheorem{construction}[theorem]{Construction}
\newtheorem{porism}[theorem]{Porism}

\DeclareMathOperator{\cofib}{cofib}
\DeclareMathOperator*{\colim}{colim}

\DeclareMathOperator{\Aut}{Aut}
\DeclareMathOperator{\Mod}{Mod}

\DeclareMathOperator{\Shv}{Shv}

\DeclareMathOperator{\QCoh}{QCoh}
\DeclareMathOperator{\KO}{KO}
\DeclareMathOperator{\End}{End}
\DeclareMathOperator{\Grp}{Grp}
\DeclareMathOperator{\Sing}{Sing}
\DeclareMathOperator{\Sp}{Sp}
\DeclareMathOperator{\Fun}{Fun}
\DeclareMathOperator{\Th}{Th}
\DeclareMathOperator{\Pic}{Pic}

\DeclareMathOperator{\CAlg}{CAlg}
\DeclareMathOperator{\GL}{GL}

\DeclareMathOperator{\Alg}{Alg}

\DeclareMathOperator{\Br}{Br}
\DeclareMathOperator{\Haunt}{Haunt}

\DeclareMathOperator{\PPr}{\mathbf{Pr}}
\DeclareMathOperator{\TTwSp}{\mathbf{TwSp}}

\def\on{\operatorname}

\DeclareMathOperator{\Hom}{Hom}
\DeclareMathOperator{\Map}{Map}

\DeclareMathOperator{\map}{map}

\DeclareMathOperator{\KU}{KU}

\DeclareMathOperator{\Open}{Open}
\DeclareMathOperator{\Cat}{Cat}
\DeclareMathOperator{\BU}{BU}
\DeclareMathOperator{\MU}{MU}

\DeclareMathOperator{\SWF}{SWF}
\DeclareMathOperator{\Spec}{Spec}

\newcommand{\et}{\mathrm{\acute{e}t}}

\newcommand{\St}{\mathrm{St}}
\newcommand{\id}{\mathrm{id}}
\newcommand{\oop}{\mathrm{op}}

\newcommand{\bC}{\mathbb{C}}
\newcommand{\bE}{\mathbb{E}}

\newcommand{\bR}{\mathbb{R}}
\newcommand{\shvcat}{\Shv_{\mathrm{Cat}_{R, \kappa}}(B)}
\newcommand{\bS}{\mathbb{S}}

\newcommand{\bZ}{\mathbb{Z}}

\newcommand{\cC}{\mathcal{C}}
\newcommand{\calD}{\mathcal{D}}
\newcommand{\cd}{\mathcal{D}}

\newcommand{\cH}{\mathcal{H}}

\newcommand{\cO}{\mathcal{O}}

\newcommand{\TwSp}{\mathrm{TwSp}}

\newcommand{\cA}{\mathcal{A}}

\newcommand{\cS}{\mathcal{S}}

\newcommand{\och}{\quad \text{and} \quad}

\author{Alice Hedenlund}
\address{Uppsala University, Sweden}
\email{alice.hedenlund@math.uu.se}

\author{Tasos Moulinos}
\address{Institute for Advanced Study, USA}
\email{tmoulinos@gmail.com}

\title{Twisted Spectra Revisited}
\date{\today}

\begin{document}

\begin{abstract}
  We recapture Douglas' framework for twisted parametrized stable homotopy theory in the language of $\infty$-categories.
  A twisted spectrum is essentially a section of a bundle of presentable stable $\infty$-categories whose fiber is the $\infty$-category of spectra, a perspective we refine in this work.
  We recover some of Douglas' results on classifications of such bundles, as well as further make precise the identification of categories of twisted spectra with module categories over Thom spectra in the pointed connected case. Furthermore, we examine subtle aspects of the functoriality which arise by virtue of being fibered over the Brauer space of the sphere spectrum.
  Extending beyond the scope of Douglas's work, we introduce a total category of twisted spectra over a fixed space, where we allow the twists to vary, and show that these total categories themselves satisfy the essential features of a 6-functor formalism.
  We also introduce an $(\infty,2)$-category of twisted spectra, and use this framework to discuss duality theory for these type of objects.
\end{abstract}

\maketitle

\setcounter{tocdepth}{1}
\tableofcontents

\section*{Introduction}

\subsection*{Background and Aim}

The motivation behind this project is grounded in what is contemporaneously known as \emph{Floer homotopy theory}.
Without going into details, let us illustrate one way to think about the Floer homotopy theory program and how twisted spectra fits into this\footnote{Considering that both of the authors of this paper identify as homotopy theorists, rather than Floer theorists, and that the paper is written with an implicit assumption that the reader is familiar with the language of $\infty$-categories, we will also assume in our explanation of Floer homotopy theory that the reader is comfortable with the \emph{homotopy theory} part, but maybe not so much with \emph{Floer} part. We apologize in advance to any Floer theorists that we will inevitably offend by this assumption.}.
One slogan for Floer theory is that it is ``a version of Morse theory for infinite-dimensional manifolds in which only a relative index can be defined''.
A rough flowchart of how Floer theory works is shown in Figure ~\ref{fig:flow_chart_Floer_htpy_thy} which we now explain in words.

\begin{figure}[h!]
  \includegraphics[scale=.65]{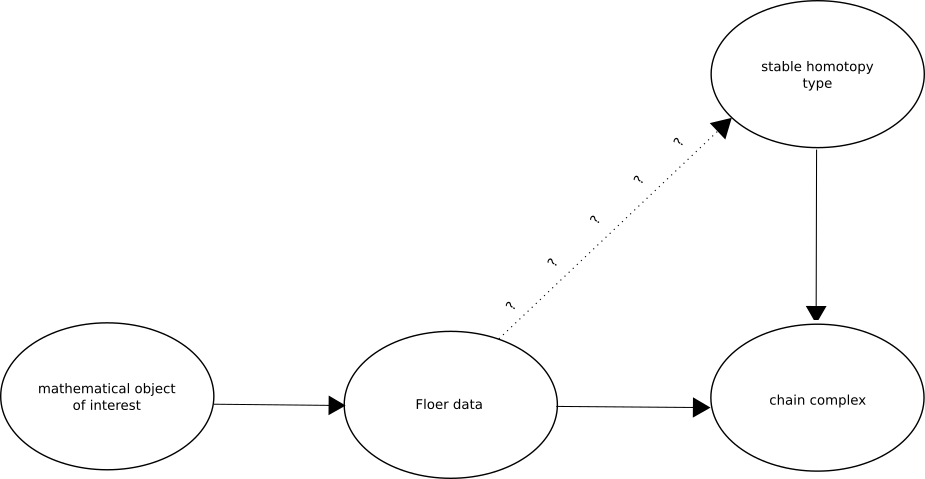}
  \caption{Possible lift of Floer homology to Floer homotopy.}
  \label{fig:flow_chart_Floer_htpy_thy}
\end{figure}
We start with some mathematical object of interest; depending on the flavor of Floer theory that we are dealing with, this could be a low-dimensional manifold, a symplectic manifold, or something else.
This is ultimately the mathematical object that we are trying to understand.
In order to understand this object, one associates to it an infinite-dimensional manifold together with some generalization of a Morse function.
Let us call this pair the \emph{Floer data}.
One then associates to the Floer data a chain complex whose homology we then refer to as \emph{Floer homology}, of whatever flavor we are working with.
For a specific example, if one is dealing with symplectic Floer theory, the mathematical object of interest would be a symplectic manifold equipped with some non-degenerate Hamiltonian and the Floer data would consist of some version of the free loop space on the manifold together with the so-called action functional~\cite{MS17}*{Chapter 11}.
From the point of view of homotopical algebra, thinking of stable homotopy theory as a generalization of homological algebra, it is now natural to ask the following question.

\begin{question}
  Rather than a chain complex, is it possible to associate a stable homotopy type to the Floer data in such a way that the homology of this stable homotopy type recovers what is classically known as Floer homology?
\end{question}

This question was first posed by Cohen--Jones--Segal~\cite{CJS95}.
For some motivation, let us just mention a few reasons for why one would be interested in this:
\begin{description}
  \item[Computational Reason] Having access to such a \emph{Floer homotopy type} would allow us define Floer invariants for other types of generalized cohomology theories, like various types of Floer K-theory and Floer bordism.
  This could give a finer understanding of the mathematical objects of interest.
  \item[Structural Reason] It is still unclear what the homotopy theoretical underpinning of Floer theory is.
  One can compare this to Morse theory of compact manifolds, where it is geometrically clear how the homotopy type of the manifold is build up using the critical points of the Morse function, and where one can compute the homology in many other ways, without even making use of the auxiliary data of a Morse function.
  In Floer theory, it is much harder to separate out what depends on the underlying homotopy types and what does not, and we do not yet know whether it is possible to define Floer homology in other ways, that do not make use of a Morse--Floer function.
  It is possible that studying Floer homotopy types might give more insight into these structural questions.
\end{description}

In the aforementioned paper, Cohen--Jones--Segal also sketch how to construct a (pro)spectrum from Floer data via flow categories in the situation that our infinite-dimensional manifold is ``trivially polarized''.
Roughly, the reader should think of a trivialization of the manifold as a decompositions of its tangent space into two parts, one part containing the positive eigenspace of the Hessian of the ``Morse function'' and another part containing the negative eigenspace, allowing for some finite-dimensional ambiguity on where to place the nullspace~\cite{CJS95,Dou05,PS86}.
Depending on the flavor of polarization, this corresponds to a class in either real or complex topological K-theory in cohomological degree $1$.
The manifold is trivially polarized when this class is zero.
In particular, this implies that the flow category is framed, from which we can define a (pro)spectrum.

\begin{question}
  What sort of objects are these ``stable homotopy types'' associated to Floer data where the polarization is \emph{non-trivial}?
\end{question}

That some twisted version of stable homotopy theory is needed to deal with non-trivial polarizations was first pointed out by Furuta in the unfinished draft~\cite{furuta-preprint}.
In his 2005 PhD thesis, Douglas proposed a more rigorous framework for these ``twisted spectra''~\cite{Dou05}.
It is these twisted spectra that this paper is concerned with.

The aim of this paper is to recast Douglas' work in the modern framework of $\infty$-categories, and to utilize this framework to prove structural results about these sort of objects.
In particular, we are interested in functoriality properties for twisted spectra, akin to various 6-functor formalisms that have appeared in the literature in other contexts.
We also want to get a better understanding of the duality theory of these types of objects.

\begin{rmk}
  We note that flow categories not being frameable is not the only issue when defining a prospectrum from Floer data.
  There is also the separate issue of ``disc bubbling'' which complicates the picture, arguably even more than the issue of frameability.
  Roughly, disc bubbling leads to issues of compactness of our flow categories.
  In most flavors of Floer theory in which disc bubbling occur, there is typically a natural way to compactify, but the precise nature of how a flow category relates to its compactification is a bit mysterious~\cite{CJS95}.
  Two common ways of dealing with bubbling issues in Floer theory today are either the theory of Kuranishi structures~\cite{FOOO20} or the theory of polyfolds~\cite{HWZ21}.
  In this paper, \emph{we will not consider bubbling phenomena at all} and strictly focus on the issue of frameability, interpreted as the need to use twisted spectra rather than just ordinary spectra.
  A true model for Floer homotopy types would be some modified version of the twisted spectra, which also takes into account bubbling issues.
\end{rmk}

\subsection*{Main Results}

Roughly, the reader should think of a twisted spectrum on a space $B$ as ``a section of a bundle of $\infty$-categories over $B$ whose fiber is the $\infty$-category of spectra''.
This is an obvious generalization of parametrized spectra: if we look at the trivial bundle $\pi_B : \Sp \times B \to B$, then sections of this is precisely functors $B \to \Sp$, that is, parametrized spectra.
To rigorously make sense of twisted spectra, we will take two perspectives: invertible sheaves of categories or local systems of categories.
For the first perspective one considers the constant sheaf of categories with value $\Sp$ on the space $B$:
\[
\cO_{B} : \Open(B)^{\oop} \longto \mathrm{Pr}^{L}_{\St} \,, \quad U \mapsto \Fun(\Sing(U),\Sp) \,.
\]
The $\infty$-category of sheaves of presentable stable $\infty$-categories on $B$ is symmetric monoidal and the above sheaf is precisely the unit of this monoidal structure.
A \emph{haunt} over a space $B$ is a locally free rank 1 module over $\cO_B$ and a \emph{twisted spectrum} over $B$ is a section of a haunt.
When our space $B$ is locally contractible, then any haunt is a locally constant sheaf, and we can use the equivalence between locally constant sheaves and local systems to conclude the following theorem, which may be compared to~\cite{Dou05}*{Proposition 2.8}.

\begin{theorem}
  Let $B$ be a locally contractible space. Then the space of haunts on $B$ is equivalent to the space of map from $\Sing(B)$ to the classifying space of the Picard space of the sphere spectrum:
  \[
  \Haunt_B \simeq \Map(\Sing(B),B\Pic(\bS)) \,.
  \]
\end{theorem}

This implies that we could have alternatively defined haunts and twisted spectra using local systems of categories. Indeed, from this point of view a haunt is simply a map\footnote{In the setting on Floer homotopy theory, these maps always seem to factor through the J-homomorphism.
Indeed, a polarization corresponds to a class in complex or real topological $K$-theory, and such classes correspond to maps $B \to U$ and $B \to U/O$, respectively. The relevant haunt is then obtained by using Bott periodicity and the J-homomorphism. A polarization being trivial (or the flow corresponding flow category being frameable) is equivalent to the map $B \to B\Pic(\bS)$ being nullhomotopic.} $\sigma : B \to B\Pic(\bS)$ and the $\infty$-category of twisted spectra for this particular haunt is
    \[
    \TwSp(B,\sigma) = \lim_{B} \left( B \overset{\sigma}\longto B\Pic(\bS) \longto \mathrm{Pr}^{L}_{\St} \right) \,.
    \]
We summarize some of our results about these categories.

\begin{theorem}
    There exists a symmetric monoidal functor
    \[
    \TwSp: \cS_{/B \Pic(\bS)} \to \on{Pr}^{L}_{\St} \,.
    \]
The functor above sends a map $f: (A, \tau) \to (B, \sigma)$ in $\cS_{/ B \Pic(\bS)}$ to the functor $f_!: \TwSp(A, \tau) \to \TwSp(B, \sigma)$. Its right adjoint $f^*$ is itself colimit-preserving and thus further admits a right adjoint $f_*$.
In particular, since the functor is symmetric monoidal, we have a natural equivalence
\[
\TwSp(A, \tau) \otimes^{L} \TwSp(B, \sigma) \simeq \TwSp(A \times B, \tau \boxtimes \sigma)
\]
which further implies an equivalence
\[
\Fun^{L}(\TwSp(A, \tau), \TwSp(B, \sigma)) \simeq \TwSp(A \times B, - \tau \boxtimes \sigma) \,.
\]
\end{theorem}

\begin{rmk}
    A key vehicle powering the theory forward at this stage, and for what comes next is the notion of ambidexterity~\cite{HL13} for diagrams in $\on{Pr}^L_{\St}$ indexed by Kan complexes.
    This more explicitly allows for one to control the homotopy coherences needed establish base-change and functoriality results.
    For more, see Section ~\ref{sec:parametrized_cats}.
\end{rmk}

  Under the assumption that $B$ is connected, these categories can be identified with module categories, a fact that was established in~\cite{Dou05}*{Proposition 3.13}, and later rediscovered in a different context in~\cite{CCRY23}*{Theorem D}.
    We reprove this fact, establishing some extra naturality properties of identification.

    \begin{theorem}
      Let $B$ be a connected space and consider a haunt $\sigma : B \to B\Pic(\bS)$.
      Then, there is a natural equivalence
      \[
      \TwSp(B,\sigma) \simeq \Mod_{\Th(\Omega \sigma)}
      \]
      of $\infty$-categories. Here, $\Th(\Omega \sigma)$ denotes the Thom spectrum of the map $\Omega \sigma : \Omega B \to \Pic(\bS)$.
    \end{theorem}

  Note that the categories $\TwSp(A,\tau)$ are not themselves monoidal; the map $\Omega \sigma$ is \emph{a priori} only $\bE_1$.
  However, one can define a total category of twisted spectra over a fixed space, where we allow the twist to vary, and this total stable $\infty$-category satisfies a 6-functor formalism.
  The following theorem is a summary of a number of different result found in Section~\ref{sec:total_category}, namely:
  Proposition~\ref{prop:total_pullback_has_right_adjoint}, Proposition~\ref{prop:total_pullback_no_left_adjoint},
  Proposition~\ref{prop:BC_total_cat},
  Proposition~\ref{prop:total_cat_symm_mon},
  Proposition~\ref{prop:pullback_strong_monoidal},
  Proposition~\ref{prop:proj_iso_total},
  and Proposition~\ref{prop:closed_total_cat}.

\begin{theorem}
  We define the total $\infty$-category of twisted spectra over a space $A$ as
  \[
  \TwSp(A) = \colim\left( \TwSp(A,-) : \Map(A,B\Pic(\bS)) \longto \mathrm{Pr}^{L}_{\St}\right) \,.
  \]
  These categories satisfy a 6-functor formalism where $f^! = f^*$. In particular:
  \begin{itemize}
    \item The categories $\TwSp(A)$ are closed symmetric monoidal.
    \item For every map $f : A \to B$ of spaces there is an induced pullback functor
    \[
    f^* : \TwSp(B) \longto \TwSp(A)
    \]
    which is symmetric monoidal. This functor admits a right adjoint
      \[
      f_* : \TwSp(A) \longto \TwSp(B) \,.
      \]
      \item Given a map of spaces $f : A \to B$ such that the induced map $\Map(B,B\Pic(\bS)) \to \Map(A,B\Pic(\bS))$ is an equivalence, there is a functor
      \[
      f_! : \TwSp(A) \longto \TwSp(B) \,.
      \]
      This functor satisfies base change and  projection formula isomorphisms.
      Moreover, the functor admits a right adjoint which coincides with $f^*$.
  \end{itemize}
\end{theorem}

The naturality of the identification between categories of twisted spectra and module categories over Thom spectra guarantee that the pullback functor and its left and right adjoints on the level of categories of twisted spectra correspond to the standard restriction, extension, and coextension of scalars on the level of module categories, in the case that we restrict ourselves to connected spaces.

By using the natural equivalence
\[
\Fun^L(\TwSp(A,\tau),\TwSp(B,\sigma)) \simeq \TwSp(A \times B , - \tau \boxtimes \sigma)
\]
already mentioned above, we can define and study dualizability of twisted spectra.
In particular, we say that a $(B,\sigma)$-twisted spectrum $X$ is \emph{Costenoble--Waner dualizable} if its corresponding colimit-preserving functor
\[
\Phi_X : \Sp \longto \TwSp(B,\sigma)
\]
is such that its right adjoint is itself colimit-preserving.
This is a direct generalization of the notion of Costenoble--Waner dualizability of parametrized spectra.
In the pointed connected case, Under the natural equivalence between categories of twisted spectra and categories of modules over Thom spectra, we note that this type of dualizability corresponds to dualizablity of bimodules.

\begin{theorem}
  Let $\alpha : G \to \Pic(\bS)$ and $\beta : H \to \Pic(\bS)$ be maps of groups over the Picard space of the sphere spectrum and suppose that we are given a $\Th(\alpha) - \Th(\beta)$-module $M$. Suppose that  $X$ is the twisted spectra corresponding to $M$ under the equivalence
  \[
  \TwSp(BG \times BH , -B\alpha \boxtimes B\beta) \simeq {}_{\Th(\alpha)} \mathrm{BiMod}_{\Th(\beta)} \,.
  \]
  Then the diagram
  \[
  \begin{tikzcd}
  \TwSp(BG, B\alpha) \arrow[r,"\Phi_X"] \arrow[d,"\simeq"] & \TwSp(BH,B\beta) \arrow[d,"\simeq"] \\
  \Mod_{\Th(\alpha)}  \arrow[r,"F_M"] & \Mod_{\Th(\beta)}
  \end{tikzcd}
  \]
  commutes where
  \[
  F_M : \Mod_{\Th(\alpha)} \longto \Mod_{\Th(\beta)} \,, \quad N \mapsto N \otimes_{\Th(\alpha)} M \,.
  \]
\end{theorem}

We conclude in Section ~\ref{derivedstory} by reinterpreting several of the notions of this paper in the context of spectral algebraic geometry.
In particular, we identify twists as derived Azumaya algebras on the constant (or Betti) stack associated to a topological space.
We then study this setup with a view towards the classical problem of Grothendieck regarding realizability of \'{e}tale cohomology classes by Azumaya algebras.

\subsection*{Terminology and Notation}

Let us collect some common terminology and notation that we will use in the paper:

\begin{itemize}
  \item As we will use $(\infty,1)$-categories extensively throughout the paper (which we will also refer to simply as $\infty$-categories), we will follow common notation and terminology for these, following~\cite{HTT, HA}.  In particular, all categorical notions should be interpreted in the $\infty$-categorical sense. For example, colimits and limits will be what would classically had been referred to as homotopy colimit and homotopy limits.
  \item We will not differentiate between a 1-category and its corresponding $\infty$-category obtained via the nerve construction.
  \item We denote by $\cS$ and $\Sp$ the $\infty$-categories of spaces\footnote{Or: ``$\infty$-groupoids'', ``Kan complexes'', ``anima'', depending on your preferred terminology.} and spectra.
  \item  Given an $\infty$-category $\cC$, the mapping space between two objects will be denoted $\Map_{\cC}(x,y)$, or just $\Map(x,y)$ if $\cC$ is implicit. If $\cC$ is stable, we will denote the mapping spectra by $\map_{\cC}(x,y)$. Recall that $\Omega^\infty\map(x,y) \simeq \Map(x,y)$. If the $\infty$-category $\cC$ has internal mapping objects, we will denote them by $\hom(x,y)$.
  \item To distinguish between $(\infty,1)$-categories and $(\infty,2)$-categories we will typically use bold font for the latter.
  For example, $\Pr^{L}_{\St}$ is the $(\infty,1)$-category of presentable stable $\infty$-categories and left adjoints, while $\PPr_{\St}^L$ is the $(\infty,2)$-categorical version.
  \item Module categories are implicitly categories of right modules unless otherwise stated.

\end{itemize}

\subsection*{Acknowledgements}

We would like to thank Ben Antieau, Stefan Behrens, Bastiaan Cnossen, Chris Douglas, David Gepner, Peter Haine, Rune Haugseng, Thomas Kragh, Achim Krause, Ciprian Manolescu, Thomas Nikolaus, Maxime Ramzi, Bertrand Toën, and Lior Yanovski, for various comments, questions, and help during the writing of this paper. Special thanks are due to Maxime Ramzi for the idea of the proof of Proposition ~\ref{prop naturality of shit}. TM would like to acknowledge the support of the Giorgio and Elena Petronio Fellowship at the Institute for Advanced Study.
We would also like to thank an anonymous referee for helpful comments on an earlier version of this document.

\section{Preliminaries}

In this section, we collect a number of well-known structures that we will work with extensively in the paper.
We recall aspects of the theory of (symmetric monoidal) stable presentable $\infty$-categories, including Picard and Brauer spaces.
We then discuss sheaves of presentable stable $\infty$-categories, and review the connection between those sheaves that are locally constant and local systems of presentable stable $\infty$-categories.

\subsection{Stable presentable $\infty$-categories}

Let us start with some preliminaries on presentable $\infty$-categories.
We refer the reader in want of more details to~\cite{HTT}*{Section 5.5} and~\cite{HA}*{Section 4.8}.
Let $\Pr^{L}$ denote the $\infty$-category of presentable $\infty$-categories and left adjoint functors.
This admits a closed symmetric monoidal structure with tensor product denoted as $\otimes = \otimes^L$ and internal mapping objects denoted $\Fun^L(-,-)$.
The latter are just colimit-preserving functors from one presentable $\infty$-category to another.
We will typically restrict our attention to those presentable $\infty$-categories that are stable, and let $\Pr^L_{\St}$ denote the full subcategory of $\Pr^L$ consisting of those.
This inherits a closed symmetric monoidal structure which we denote by the same notation as before.
The tensor product is sometimes referred to as the \emph{Lurie tensor product}.
The unit for the symmetric monoidal structure is $\Sp$, the stable $\infty$-category of spectra.
Commutative algebra objects in $\Pr^{L}_{\St}$ are precisely symmetric monoidal presentable stable $\infty$-categories whose tensor product preserves small colimits separately in each variable, also called \emph{presentably symmetric monoidal $\infty$-categories}.

\begin{theorem}[Universal Property of the Lurie Tensor Product] \label{thm:uni_prop_tensor_PrLSt}
If $\cC$ and $\cd$ are presentable $\infty$-categories, then the tensor product $\cC \otimes \cd$ is the universal presentable $\infty$-category receiving a functor $\cC \times \cd \to \cC \otimes \cd$ which preserves small colimits separately in each variable.
\end{theorem}

If $R$ is an $\bE_\infty$-ring, then $\Mod_R$, being a presentably symmetric monoidal stable $\infty$-category, is a commutative algebra object in $\Pr^{L}_{\St}$.
Hence, we can make sense of modules over $\Mod_R$ using the Lurie tensor product.

\begin{definition}
  Let $R$ be an $\bE_\infty$-ring. The $\infty$-category of \emph{$R$-linear categories} is defined as
  \[
  \Cat_R = \Mod_{\Mod_R}\left(\mathrm{Pr}^{L}_{\St}\right) \,.
  \]
\end{definition}

Note that $\Cat_R$ is itself a closed symmetric monoidal $\infty$-category with unit $\Mod_R$. By slight abuse of notation, we will denote the tensor product in $\Cat_R$ by $\otimes_R$.
Note that $\Cat_{\bS}$ is canonically identified with $\Pr^{L}_{\St}$ itself.
Before we move on, here is a assortment of various remarks that we will use later on.

\begin{rmk}
  If $R$ is an $\bE_n$-ring, then $\Mod_R$ is an $\bE_{n-1}$-algebra object in $\Pr^{L}_{\St}$, and $\Cat_R$ is an $\bE_{n-2}$-monoidal category.
  Hence, if we want $\Cat_R$ to be a monoidal category, then we need to assume that $R$ is at least an $\bE_3$-ring.
  All of the examples of $R$ that we have in mind in this paper will be $\bE_\infty$-rings though, so that $\Mod_R$ and $\Cat_R$ are themselves symmetric monoidal $\infty$-categories.
\end{rmk}

\begin{rmk}
 Recall that a presentable $\infty$-category is $\kappa$-compactly generated for some infinite regular cardinal~$\kappa$, by definition.
 For technical reasons, we will sometimes fix a sufficiently large uncountable regular cardinal ~$\kappa$, and work with presentable categories which are $\kappa$-compactly generated.
 One can form an $\infty$-category~$\Pr^{L}_{\kappa}$ of such categories.
 The point here is that while $\Pr^{L}_{\St}$ is not itself presentable, the $\infty$-category $\Pr^{L}_{\St,\kappa}$ is \cite{HA}*{Lemma 5.3.2.9.}.
 Since~$\Pr^L_{\St}$ is a filtered colimit of the presentable $\infty$-categories $\Pr_{\St,\kappa}^{L}$, this means that a lot of set-theoretical difficulties can be dealt with by using minor modifications.
 If we work over an $\bE_\infty$-ring $R$, as opposed to the sphere spectrum, we will use the notation
\[
\Cat_{R,\kappa} = \Mod_{\Mod_R}(\mathrm{Pr}^{L}_{\St,\kappa}) \,.
\]
We denote by $\Cat_{R, \omega}$ the $\infty$-category of \emph{compactly generated $R$-linear categories}.
If $R= \bS$, we retain the use of the notation $\Pr^{L}_{\St, \omega}$.
The $\infty$-category $\Cat_{R, \omega}$ inherits a symmetric monoidal structure from $\Cat_R$.
\end{rmk}

\begin{rmk}
  There is also an $\infty$-category $\Pr^R$ which consists of presentable $\infty$-category and right adjoint functors.
  There is an equivalence
  \[
  (\mathrm{Pr}^L)^{\oop} \simeq \mathrm{Pr}^{R}
  \]
  of $\infty$-categories which is the identity on objects and sends a left adjoint to its right adjoint on the level of morphisms.
  We will write $\Pr^{R}_{\St}$ for the full subcategory of $\Pr^R$ consisting of those presentable $\infty$-categories that are stable.
  We will also consider the $\infty$-category $\Pr^{L,R}$ consisting of presentable $\infty$-categories and functors that are \emph{both} left and right adjoints.
  The full subcategory that consists of those presentable $\infty$-categories that are stable will again be denoted $\Pr^{L,R}_{\St}$.
\end{rmk}

\subsection{Picard and Brauer Spaces}

For a symmetric monoidal $\infty$-category $\cC$, we let $\Pic(\cC)$ denote the $\infty$-groupoid of invertible objects in $\cC$ and the equivalences between them.
We refer to this as the \emph{Picard space} of $\cC$.
Note that the Picard space is closed under the symmetric monoidal structure, so that inherits the structure of an $\bE_\infty$-space.
Moreover, this $\bE_\infty$-structure is grouplike.

\begin{definition}
If $R$ is an $\bE_\infty$-ring, we will write
\[
\Pic(R) = \Pic(\Mod_{R})
\]
and refer to this as the \emph{Picard space} of $R$.
\end{definition}

We note that $\Pic(R)$ is essentially a delooping of the space of units $\GL_1(R)$.
In particular, there is an equivalence
\[
\Pic(R) \simeq \pi_0 \Pic(R) \times B\mathrm{GL}_1(R)
\]
as spaces.
As we remarked in the previous section, if $R$ is an $\bE_\infty$-ring, then the $\infty$-category of compactly generated $R$-linear categories is also symmetric monoidal, so we can make the following definition.

\begin{definition} \label{def:Brauer_space}
If $R$ is an $\bE_\infty$-ring, we will write
\[
\Br(R) = \Pic(\Cat_{R,\omega})
\]
and refer to this as the \emph{Brauer space} of $R$.
\end{definition}

The following result can be found in the literature; see for example~\cite{GL21}*{Proposition 5.7}.

\begin{proposition}
  The inclusion of the identity component gives us a map $B\Pic(R) \to \Br(R)$ which induces an equivalence
  \[
  \Pic(R) \simeq \Omega B \Pic(R) \overset{\simeq}\longto \Omega \Br(R)
  \]
  on applying the loop space functor.
\end{proposition}

\begin{rmk}
  For $R=\bS$, the inclusion $B\Pic(\bS) \to \Br(\bS)$ is already an equivalence.
\end{rmk}

The homotopy groups of the Brauer space $\Br(R)$ agree with the shifted homotopy groups of $R$ in high enough dimensions and reduces to {\'e}tale cohomology invariants in low dimensions.

\begin{proposition}\cite{AG14}*{Corollary 7.13}
Let $R$ be a connective $\bE_\infty$-ring. Then we have
\[
\pi_{n} \Br(R) = \begin{cases}
H^1_{\et}(\Spec \pi_0 R , \bZ) \times H^2_{\et}(\Spec \pi_0 R, \mathbb{G}_m) & n=0 \\
H^0_{\et}(\Spec \pi_0 R , \bZ) \times H^1_{\et}(\Spec \pi_0 R, \mathbb{G}_m) & n=1 \\
(\pi_0 R)^\times & n=2 \\
\pi_{n-2} R & n \geq 3 \,.
\end{cases}
\]
In particular, note that if $R$ is a classical discrete commutative ring, then the Brauer space has at most three non-trivial homotopy groups.
\end{proposition}

For reference, we list some low dimensional homotopy groups for various common $\bE_\infty$-rings in Table~\ref{table:homotopy}. We assembled this table from various sources, including~\cite{Dou05}*{Table 3},~\cite{AG14}*{Theorem 7.16}, and~\cite{DK70}.

\begin{table}[h!]
\begin{center}
  \begin{tabular}{ c | c c c c c c c}
    $\pi_n \Br(R)$ & $0$ & 1 & 2 & 3 & 4 & 5 & 6 \\
    \hline
    $\bS $ & 0 & $\bZ$ & $\bZ/2$ & $\bZ/2$ & $\bZ/2$ & $\bZ/24$ & 0 \\
    $H\bZ$ & 0 & $\bZ$ & $\bZ/2$ & 0 & 0 & 0 & 0 \\
    $Hk$ & 0 & $\bZ$ & $k^{\times}$ & 0 & 0 & 0 & 0 \\
    $\MU$ & 0 & $\bZ$ & $\bZ/2$ & 0 & $\bZ$ & 0 & $\bZ$ \\
    $\KU$ & ?? & $\bZ/2$ & $\bZ/2$ & 0 & $\bZ$ & 0 & $\bZ$ \\
    $\KO$ & ?? & $\bZ/8$ & $\bZ/2$ & $\bZ/2$ & $\bZ/2$ & 0 & $\bZ$
  \end{tabular}
  \caption{Table of homotopy groups $\pi_n \Br(R)$ for various common $\bE_\infty$-rings $R$. Here, $k$ is an arbitrary field.}
  \label{table:homotopy}
\end{center}
\end{table}

\subsection{The $(\infty, 2)$-category of presentable stable $\infty$-categories}

The $(\infty,1)$-category of stable presentable $\infty$-categories is simply a shadow of higher structure.
Indeed, note that given any two presentable $\infty$-categories~$\cC$ and $\calD$, there exists an internal mapping object $\Fun^{L}(\cC, \calD)$, the $\infty$-category of left adjoint functors between them.
This suggests that there is a natural enrichment of $\Pr^{L}_{\St}$ in $\infty$-categories, giving it the structure of an $(\infty, 2)$-category~\cite{GH15}.
A detailed description of the current status of the theory of $(\infty,2)$-categories would take us very far afield, and is not strictly necessary for results of this paper.
There are several models for $(\infty,2)$-categories at this point, but we will typically think of these as enriched over $\infty$-categories, the same way as $\infty$-categories are enriched over $\infty$-groupoids.

\begin{definition}
By~\cite{GR17}*{Chapter I.1}, there exists a symmetric monoidal $(\infty,2)$-category consisting of stable cocomplete $\infty$-categories. We let $\mathbf{Pr}^{L}_{\St}$ denote the full sub-($\infty,2)$-category of the above, spanned by  the presentable stable $\infty$-categories.
\end{definition}

By the discussion in~\cite{HSS17}*{Section 4.4}, for any $(\infty,2)$-category $\mathbf{C}$ and any $\bE_{\infty}$-algebra $R$ in the core $\infty$-category $\iota_1 \mathbf{C}$, there exists an $(\infty,2)$-category $\mathbf{Mod}_{R}$ for which
\[
\iota_1 \mathbf{Mod}_{R} \simeq \Mod_R(\iota_1 \mathbf{C}) \,.
\]
Applying this  to  $\mathbf{C} = \mathbf{Pr}^{L}_{\St}$, we obtain an $(\infty,2)$-category organizing the totality of $R$-linear categories, which we denote $\mathbf{Cat}_{R}$ again using bold notation as.

\begin{definition}
As $\Cat_{R, \omega}$ is an full subcategory of $\Cat_{R}$, and since the latter has an $(\infty,2)$-categorical enhancement, we conclude that compactly generated $R$-linear categories form an $(\infty,2)$-category as well, which we denote by $\mathbf{Cat}_{R, \omega}$.
\end{definition}

\subsection{Sheaves of stable presentable $\infty$-categories}

Some of the main objects of this paper will arise via sheaves on a fixed topological space valued in presentable stable $\infty$-categories.
In this section, we go over what this means.
First, we recall what a sheaf on an arbitrary (small) $\infty$-category $\cC$ is.
We refer to~\cite{HTT}*{Chapter 5} and~\cite{SAG}*{Section 1.3.1.} for details.

\begin{definition}
    Let $\cC$ be a $\kappa$-small $\infty$-category. Then we define the $\infty$-category of \emph{presheaves on $\cC$} to be the functor category
    \[
    \mathrm{PSh}(\cC) = \Fun(\cC^{\oop}, \cS).
    \]
\end{definition}

Suppose that $\cC$ comes equipped with a Grothendieck topology.
In this case, one might want to impose a descent condition with respect to the chosen topology.
The language of localizations of $\infty$-categories is well-suited for describing this passage.

\begin{definition}
    We define the $\infty$-category of \emph{sheaves on $\cC$} as the left-exact accessible localization
\[
\Shv(\cC) \lhook\joinrel\longrightarrow \mathrm{PSh}(\cC) \,.
\]
This is the full subcategory of functors which satisfy descent with respect to covers given by the Grothendieck topology on $\cC$.
\end{definition}

\begin{rmk}
As  $\mathrm{PSh}(\cC)$ and $\Shv(\cC)$ are presentable $\infty$-categories, we may stabilize by tensoring along
$$
- \otimes \Sp: \mathrm{Pr}^{L} \longto \mathrm{Pr}^{L}_{\St}
$$
to obtain $\infty$-categories of presheaves and sheaves on $\cC$ valued in the $\infty$-category of spectra.
We use the notation  $\mathrm{PSh}_{\Sp}(\cC)$ and $\Shv_{\Sp}(\cC)$ for these respective stable $\infty$-categories.
\end{rmk}

Now we specialize, essentially once and for all, on the case $\cC = \Open(B)$, the category of open subsets of a fixed topological space $B$, with morphisms given by inclusions of open sets.
By identifying this with its nerve, we view this as an $\infty$-site, with Grothendieck topology determined by the usual topology on $B$.
Recall that a continuous map $f : B \to A$ of topological spaces gives rise to a restriction morphism $\Open(A) \to \Open(B)$ of sites determined by the assignment $V \mapsto f^{-1}(V)$.
We now introduce the $\infty$-category of sheaves of $\kappa$-compactly generated $R$-linear categories on a fixed topological space $B$.

\begin{definition}
    \leavevmode
    \begin{itemize}
      \item We denote the $\infty$-category of $\mathrm{Cat}_{R, \kappa}$-valued presheaves by
      \[
   \mathrm{PSh}_{\mathrm{Cat}_{R, \kappa}}(B) = \Fun(\Open(B)^{\oop},\mathrm{Cat}_{R, \kappa} ) \,.
      \]
      \item We denote the full subcategory of $\Cat_{R,\kappa}$-valued sheaves with respect to the aforementioned Grothendieck topology  by
       \[
       \Shv_{\Cat_{R, \kappa}}(B) = \Fun^{\tau }( \Open(B)^{\oop},\mathrm{Cat}_{R, \kappa}) \,.
       \]
       The subscript $\tau$ on notation on the right denotes the accessible localization of the $\infty$-category of presheaves corresponding to imposing descent along covers in the topology of $B$.
    \end{itemize}
\end{definition}

We shall refer to objects of $\Shv_{\mathrm{Cat}_{R, \kappa}}(B)$ as \emph{$R$-linear categorical sheaves on $B$}.
Recall that the fully faithful inclusion of sheaves into presheaves admits a left adjoint
\[
L : \mathrm{Psh}_{\mathrm{Cat}_{R, \kappa}}(B)  \longto \Shv_{\mathrm{Cat}_{R, \kappa}}(B)
\]
that is referred to as \emph{sheafification}.
This exhibits $\Shv_{\mathrm{Cat}_{R, \kappa}}(B)$  as an accessible localization.

\begin{rmk} \label{alternatedescriptionofsheaves}
Let $\mathcal{X}$ denote the $\infty$-topos of sheaves of spaces on $B$.
Then by ~\cite[Proposition 4.8.1.17]{HA}, one has the description
$$
\Shv_{\mathrm{Cat}_{R, \kappa}}(B) \simeq \mathcal{X} \otimes \mathrm{Cat}_{R, \kappa} \simeq \Fun^{R}( \mathcal{X}^{\oop}, \mathrm{Cat}_{R, \kappa})
$$
where the middle and left terms denote the Lurie tensor product of $\mathcal{X}$ and $\mathrm{Cat}_{R, \kappa}$  and the presentable $\infty$-category of right adjoint functors between $\mathcal{X}$ and $\mathrm{Cat}_{R, \kappa}$.
This uses presentability of $\mathrm{Cat}_{R, \kappa}$.
\end{rmk}

The category $\mathrm{PSh}_{\mathrm{Cat}_{R, \kappa}}(B)$ comes equipped with the pointwise symmetric monoidal structure.
Note that the unit of this symmetric monoidal structure is the constant functor with value $\Mod_R$.
The sheafification $L$ is compatible with the symmetric monoidal structure, so that $\Shv_{\mathrm{Cat}_{R, \kappa}}(B)$ itself inherits a symmetric monoidal structure, simply given as the sheafification of the pointwise tensor product~\cite[Proposition 1.3.4.6]{SAG}.
The sheafification functor is then a symmetric monoidal functor.
We would like to understand the unit of this symmetric monoidal structure.

\begin{definition}
  The \emph{structure sheaf of parametrized $R$-modules} is the sheaf
  \[
  \cO_{B,R} : \Open(B)^{\oop} \longto \Cat_{R} \,, \quad U \mapsto \Fun(\mathrm{Sing}(U),\Mod_R) \,.
  \]
\end{definition}

We remark that the structure sheaf is manifestly a sheaf because if $ \mathcal{U}_\bullet \to U$ is some open cover of $U$, then we have the equivalence
\[
\cO_{B,R}(U) = \Fun(U , \Mod_R)  \simeq \Fun (\colim_i \mathcal{U}_i, \Mod_R) \simeq \lim_i \Fun( \mathcal{U}_i, \Mod_R)  = \lim_i \cO(U_i) \,.
\]
We claim that this sheaf is equivalent to the sheafification of the constant functor.

\begin{proposition}
There is an equivalence $L(\mathrm{const}_{\Mod_R}) \simeq  \cO_{B,R}$. Hence $\cO_{B,R}$ is the unit for the symmetric monoidal structure on $\Shv_{\mathrm{Cat}_{R, \kappa}}(B)$.
\end{proposition}

\begin{proof}
We first claim that there exists a map of presheaves $\mathrm{const}_{\Mod_R} \to  \cO_{B,R}$, which for each open set $U$ corresponds to a functor $\Mod_R \to \Fun(\Sing(U), \Mod_R)$
Upon taking adjoints, this gives rise to a map $L(\mathrm{const}_{\Mod_R}) \to  \cO_{B,R}$
in $\Shv_{\Cat_{R,\kappa}}(B)$ which we would like to show is an equivalence.
As usual, it is enough to check this on stalks.
Of course, the stalk of $L(\mathrm{const}_{\Mod_R}) $ is exactly $\Mod_R$, as it is determined by the stalks of the  underlying presheaf.
It is easy to see that the stalk of $\cO_{B,R}$ at some point is $\Mod_R$, as well.
Thus the induced map $L(\mathrm{const}_{\Mod_R}) \to  \cO_{B,R}$ is an equivalence on stalks, and since taking stalks for all points is a conservative functor, this shows that the map is itself an equivalence.
\end{proof}

\subsection{Parametrized categories and ambidexterity}
\label{sec:parametrized_cats}

Before specializing to the situation at hand, we start with a basic treatment on parametrized stable $\infty$-categories in general.
In contrast to the last section, note that $B$ will be an object in the $\infty$-category of spaces, rather than a topological space, in what follows.

\begin{definition}
Let $B$ be a space and let $R$ denote a fixed $\bE_\infty$-ring.
We define a \emph{parametrized $R$-linear category} to be a functor $F: B \to  \Cat_{R}$.
\end{definition}

Such gadgets naturally assemble to form an $\infty$-category $\Fun(B, \Cat_{R})$ which we refer to as the $\infty$-category of parametrized $R$-linear categories on $B$.
We remark that this inherits a pointwise symmetric symmetric monoidal structure from the symmetric monoidal structure on $ \Cat_{R}$.
Note that if $F: B \to \Cat_{R}$ is a parametrized $R$-linear category then, as $\Cat_R$ is both complete and cocomplete, we may form the limit and colimit of $F$.
In fact, these two will agree, due to the phenomenon known as ambidexterity.
This is not a new result: it appears for example in~\cite{HL13}*{Example 4.3.11.}.
However, for ease of reference, we state it below with a proof.

\begin{proposition}[Ambidexterity] \label{ambidexterity in Prl}
Given a parametrized $R$-linear category $F : B \to \Cat_R$ there is a natural equivalence
\[
\colim_B F \simeq \lim_B F \,.
\]
\end{proposition}

\begin{proof}
Without loss of generality, let us assume that $R = \bS$.
The limit of the parametrized $R$-linear category can be computed in $\Cat_\infty$ since the inclusion $\Pr^{L}_{\St} \to \Cat_{\infty}$ preserves all small limits~\cite{HTT}*{Proposition 5.5.3.13.}.
Now, the colimit of the diagram $F$ can be computed as the limit along the opposite diagram
\[
F^{\oop}: B^{\oop} \to (\mathrm{Pr}^{L}_{\St})^{\oop} \simeq \mathrm{Pr}^{R}_{\St} \,.
\]
Again in this case, the inclusion $\Pr^{R}_{\St} \to \Cat_{\infty}$ preserves all small limits~\cite{HTT}*{Theorem 5.5.3.18}, so in order to compute the colimit of $F$, we find ourselves computing the limit of the diagram $B^{\oop} \to \Pr^{R}_{\St} \to \Cat_\infty$.
Since $B$ is a Kan complex, there is an equivalence $\alpha : B \simeq B^{\oop}$, identifying each path $\iota^{\oop} : x_1 \to x_0$ in $B^{\oop}$ with $\iota^{-1}: x_1 \to x_0$ in $B$.
Since the functor $F$ factors through the core groupoid of $\mathrm{Pr}^{L}_{\St}$, and the natural diagram
\[
\begin{tikzcd}
  B \arrow[r] \arrow[d,"\simeq"] & (\mathrm{Pr}^L_{\St})^{\simeq} \arrow[r] \arrow[d,"\simeq"] & \mathrm{Pr}^{L}_{\St} \arrow[r] & \Cat_{\infty} \\
  B^{\oop} \arrow[r] & ((\mathrm{Pr}^{L}_{\St})^{\simeq})^{\oop} \arrow[r] & \mathrm{Pr}^R_{\St} \arrow[ur]
\end{tikzcd}
\]
commutes, we are able to identify the diagrams $F$ and $F^{\oop}$ in $\Cat_{\infty}$.
Putting all this together, we obtain the equivalence as wanted.
\end{proof}

As a consequence of Proposition \ref{ambidexterity in Prl}, we will unambiguously refer to the colimit and the limit of a parametrized $R$-linear category as the \emph{global sections}.
If we have no reason to worry about whether we are referring to the colimit or the limit, we will denote the global sections functor by
\[
\Gamma : \Fun(B,\Cat_R) \longto \Cat_R \,.
\]
As an important consequence of the above ambidexterity result, we may identify the following adjoint pairs on global sections.
For ease of notation, we will abusively denote objects in the $\infty$-category $\Gamma F$ by $(x_b)_{b \in B}$ where $x_b \in F b$.

\begin{lemma} \label{lem:parametrized_cats_mapping_space}
  Let $F : B \to \Pr^{L}_{\St}$ be a parametrized category. The mapping space of $\Gamma F$ is given by
  \[
  \Map_{\Gamma F}((x_b)_{b \in B},(y_b)_{b \in B}) \simeq \lim_{b \in B} \Map_{Fb}(x_b,y_b) \,.
  \]
  \begin{proof}
    The mapping space in the category $\Gamma F$ is described via the pullback diagram
    \[
    \begin{tikzcd}
      \Map_{\Gamma F}((x_b)_{b \in B},(y_b)_{b \in B}) \arrow[r] \arrow[d] & \Fun(\Delta^1 , \Gamma F) \arrow[d,"{(d_1,d_0)}"] \\
      \{\ast\} \arrow[r,"{((x_b)_{b \in B},(y_b)_{b \in B})}"] & (\Gamma F)^{\oop} \times \Gamma F \,.
    \end{tikzcd}
    \]
    By using the description of $\Gamma$ as a limit and the appropriate commutativity of limits with pullbacks and products, we obtain the result.
  \end{proof}
\end{lemma}

\begin{proposition} \label{prop:colim_preserves_adjoints}
  Let $B$ be a Kan complex and consider the functor
  \[
  \Gamma : \Fun(B,\mathrm{Pr}^{L}_{\St}) \longto \mathrm{Pr}^L_{\St}
  \]
  and suppose we are given natural transformations $\lambda : F \to G$ and $\rho : G \to F$ of functors $F,G : B \to \mathrm{Pr}^{L}_{\St}$ such that $\lambda_b : Fb \to Gb$ is left adjoint to $\rho_b : Gb \to Fb$ for all $b \in B$. Then
  \[
  \Gamma \lambda : \Gamma F \longto \Gamma G \quad \dashv \quad \Gamma \rho : \Gamma G \longto \Gamma F \,.
  \]
  \begin{proof}
    Let us think of and denote objects in the categories $\Gamma F$ and $\Gamma G$ as families $(x_b)$ where $x_b \in Fb$ and $(y_b)$ where $y_b \in Gb$. Using this notation, we have that
    \[
    \Gamma \lambda : \Gamma F \longto \Gamma G \,, \quad (x_b)_{b \in B} \mapsto (\lambda_b(x_b))_{b \in B}
    \]
    Using the description of mapping spaces in the categories $\Gamma F$ and $\Gamma G$, as well as the adjunction properties we obtain
    \begin{align*}
      \Map_{\Gamma G}((\Gamma \lambda)((x_b)_{b \in B}),(y_b)_{b \in B}) &\simeq \lim_{B} \Map_{Gb}(\lambda_b(x_b),y_b) \\ &\simeq \lim_B \Map_{Fb}(x_b,\rho_b(y_b)) \\ &\simeq  \Map_{\Gamma F}((x_b)_{b \in B},(\Gamma \rho)((y_b)_{b \in B}))
    \end{align*}
    which finishes the proof.
    \end{proof}
\end{proposition}

\begin{proposition}
Let $X$ be a Kan complex and consider the functor
\[
\colim_X \simeq \lim_X : \Fun(X,\mathrm{Pr}^{R}_{\St}) \longto \mathrm{Pr}^R_{\St} \,,
\]
and suppose we are given natural transformations $\lambda : F \to G$ and $\rho : G \to F$ of functors $F,G : X \to \Pr^{R}_{\St}$ such that $\lambda_x : F x \to Gx$ is left adjoint to $\rho_x : Gx \to Fx$ for all $x \in X$. Then
\[
\lim_X \lambda : \lim_X F \longto \lim_X G \quad \dashv \quad \lim_X \rho : \lim_X G \longto \lim_X F \,.
\]
\begin{proof}
  Similar to before.
\end{proof}
\end{proposition}

As corollaries to ambidexterity, we also have the following results, that we will make heavy use of in the latter sections of this paper.

\begin{proposition} \label{prop:left_right_adjoints_ambidexterity}
    Let $f : A \to B $ be a map of spaces and let $F:   B \to \Pr^{L}_{\St}$ be a parametrized category. The following holds:
    \begin{enumerate}
      \item There is a functor
      \[
      f^*: \lim_{B}F \longto \lim_{A}f^* F
      \]
      which is a left adjoint.
      \item There is a functor
      \[
      f_!: \colim_{A} f^* F \longto \colim_{B} F
      \]
      which is a left adjoint.
      Moreover, under ambidexterity its right adjoint can be identified with the functor $f^*$ above.
    \end{enumerate}
    \begin{proof}
      \leavevmode
      \begin{enumerate}
        \item Given a limiting cocone $B^{\triangleleft} \to \Pr^{L}_{\St}$ extending $F$, there is canonically an object of ${\Pr^{L}_{\St}}_{/f^* F}$, so that we get a map $B^{\triangleleft} \to A^{\triangleleft} \to \Pr^{L}_{\St}$ by the universal property of $A^{\triangleleft} \to \Pr^{L}_{\St}$ as the terminal object in ${\Pr^{L}_{\St}}_{/f^* F}$.
        Hence we get a map
        \[
        f^* : \lim_B F \longto \lim_A f^* F
        \]
        in $\Pr^{L}_{\St}$, which is hence a left adjoint.
        \item The construction of $f_!$ is essentially dual to the construction of $f^*$ above, using functoriality of the colimit construction. What is left to show is that $f_!$ is the left adjoint to $f^*$ constructed above, under ambidexterity. Since $f_!$ is a map in $\Pr^{L}_{\St}$ it must have a right adjoint. This right adjoint is obtained by passing to the opposite category $(\Pr^{L}_{\St})^{\oop} \simeq \Pr^{R}_{\St}$ where we obtain a map
        \[
        \lim_B F \longto \lim_A f^* F \,.
        \]
        This map satisfies the same universal property as the map $f^*$ constructed above, hence they must be naturally equivalent.
      \end{enumerate}
    \end{proof}
\end{proposition}

Note that the above theorem provides us with functors
\[
\colim : \cS_{/\Pr^{L}_{\St}} \longto \mathrm{Pr}^{L}_{\St} \och \lim : (\cS_{/\Pr^{L}_{\St}})^{\oop} \longto \mathrm{Pr}^{L}_{\St} \,,
\]
where the colimit functor sends the map $f : A \to B$ in $\cS_{/\Pr^{L}_{\St}}$ to the functor $f_!$ and the limit functor sends it to~$f^*$.
For reference, let us state the following results regarding monoidality of the first functors.

\begin{lemma}\label{lem:colim_symm_monoid}
  The functor
  \[
  \colim : \cS_{/\Pr^{L}_{\St}} \longto \mathrm{Pr}^{L}_{\St}
  \]
  is symmetric monoidal.
  Here, the symmetric monoidal structure on the source category is obtained by intertwining the cartesian product on $\cS$ with the Lurie tensor product.
  In particular, if $F : A \to \Pr^{L}_{\St}$ and $G : B \to \Pr^{L}_{\St}$ are two parametrized categories, then their tensor product is the parametrized category
  \[
  F \boxtimes G : A \times B \overset{F \otimes G}\longto \mathrm{Pr}^{L}_{\St} \times \mathrm{Pr}^{L}_{\St} \overset{\otimes}\longto \mathrm{Pr}^{L}_{\St} \,.
  \]
  \begin{proof}
    Follows from the Lurie tensor product preserving small colimits.
  \end{proof}
\end{lemma}

\subsection{Locally constant sheaves and parametrized objects}
Let $\cC$ be a presentable $\infty$-category, which for the sake of specificity, the reader may take to be $\Cat_{R, \kappa}$.
We now invoke a recurring theme in the story, namely that the homotopy theory of \emph{locally constant sheaves} of objects in $\cC$ agrees with the homotopy theory of $B$-parametrized objects of $\cC$.
After recalling the relevant notions, we review this correspondence.

\begin{definition}
    Let $B$ be a topological space and let $\cC$ be a presentable $\infty$-category.
    \begin{itemize}
      \item We say that a sheaf $\mathcal{F} : \Open(B)^{\oop} \to \cC$ is \emph{constant} if it lies in the image of the canonical morphism $\cC \to \Shv_{\cC}(B)$ induced by the terminal geometric morphism $\cS \to \Shv_{\cC}(B)$ of $\infty$-topoi.
      \item We say $\mathcal{F}$ is \emph{locally constant} if there is an jointly surjective open cover $ \amalg \phi_{\alpha}: \amalg_{\alpha} U_{\alpha} \to B$, such that for every $\alpha, \phi_{\alpha}^*( \mathcal{F})$  is  constant in $\Shv_{\cC}(U_\alpha)$.
    \end{itemize}
We denote the full subcategory of $\Shv_{\cC}(B)$ consisting of those sheaves that are locally constant by $\Shv_{\cC}(B)^{\mathrm{lc}}$.
\end{definition}

\begin{theorem} \label{thm:sheaves_vs_local_systems}
Let $B$ be a locally contractible topological space\footnote{All CW-complexes and paracompact manifolds are locally contractible, so this is a very large class of topological spaces, which encompasses any applications that we have in mind.}.
Then there is an equivalence
$$
\shvcat^{\mathrm{lc}} \simeq \Fun(\Sing(B), \Cat_{R, \kappa}) \,.
$$
\end{theorem}

\begin{proof}
A proof of this fact in the context of sheaves of spaces may be found in~\cite{HA}*{Theorem A.1.15}.
In particular, it is there shown that there is an equivalence
$$
\cS_{/B} \simeq \Shv(B)^{\mathrm{lc}}
$$
of parametrized spaces over $B$ with locally constant sheaves of spaces on $B$.
We can tensor this equivalence with the presentable $\infty$-category $\cC$ using the symmetric monoidal structure of $\Pr^{L}$.
Using the fact that there are equivalences
$$
\cS_{/B} \otimes^L \cC\simeq \Fun(B, \cC) \quad \Shv(B)^{\mathrm{lc}} \otimes^{L} \cC \simeq \Fun^{R}(\Shv(B)^{\mathrm{lc},\oop}, \cC)\simeq \Shv_{\mathrm{Cat}_{R, \kappa}}(B)^{\mathrm{lc}} \,,
$$
we obtain the desired equivalence by setting $\cC = \Cat_{R, \kappa}$.
\end{proof}

\begin{rmk}
The above theorem holds true even if we replace  $\Cat_{R, \kappa}$ with any $\infty$-category $\cC$ admitting small limits, with a different proof.
\end{rmk}

\begin{rmk}
We remind the reader that $\cC$-valued sheaves are not homotopy invariant.
That is to say that $\Shv_{\cC}(B) \nsimeq  \Shv_{\cC}(A)$ even though $B \simeq A$.
Specializing to our context where $\cC = \Cat_{R, \kappa}$, and noting that $\Fun(B, \Cat_R)$ is manifestly homotopy invariant, we interpret the above as saying that $\infty$-categories of locally constant sheaves of categories are themselves homotopy invariant.
\end{rmk}

\section{Haunts}

In this section, we introduce the notion of a  haunt.
These were first introduced by Douglas~\cite{Dou05}.
Roughly, a haunt over a space $B$ is to be thought of as a bundle of $\infty$-categories over $B$ whose fiber is the $\infty$-category of spectra.
We will consider haunts mainly from two perspectives: 1) sheaves of categories and 2) local systems of categories.
Having the two perspectives on haunts in mind is usually quite useful.
Thinking about haunts in terms of sheaves is well-suited for constructions and geometric intuition, while thinking about haunts in terms of local systems comes in handy for abstract properties and proving theorems.
In the last subsection, we look at explicit examples of haunts over a few different spaces, as well as examples of haunts that naturally appear in the context of various flavors of Floer theory.

\subsection{PoV I: Invertible Sheaves of Categories}

Let us fix a topological space $B$ and an $\bE_\infty$-ring $R$.
Recall that the structure sheaf of $R$-modules parametrized over $B$ is the sheaf of categories
\[
\cO_{B,R} : \Open(B)^{\oop} \longto \Cat_R \,, \quad U \mapsto \Fun(\Sing(U),\Mod_R) \,.
\]
This is also the sheafification of the constant functor with value $\Mod_R$, and is hence the unit of the symmetric monoidal structure on $\Shv_B(\Cat_{R,\kappa})$.

\begin{definition} \label{defn:locallyfreernk1}
Let $\mathcal{F}$ be sheaf of $R$-linear categories over $B$.
We say $\mathcal{F}$ is \emph{locally free of rank 1}, if for every point in $B$ there is an open set $U$ containing the point, for which $\mathcal{F}(U) \simeq \cO_{B, R}(U)$.
\end{definition}

\begin{definition} \label{defn:haunt}
 We will refer to a locally free rank 1 modules over $\cO_{B,R}$ as an $R$-\emph{haunt} over $B$.  Let us set the notation
 $\Haunt_{B,R}$ 
to denote the space of $R$-haunts on $B$.
If $R = \bS$, we will drop the $\bS$ and simply call these \emph{haunts}.
\end{definition}

\begin{rmk}
  We could have more generally defined haunts as invertible sheaves of categories, namely, as objects in the Picard $\infty$-groupoid
  \[
  \Pic\left(\Mod_{\cO_{B,R}}(\Shv_{\mathrm{Cat}_{R, \kappa}}(B))\right) \,.
  \]
  Every locally free rank $1$ module is invertible with respect to the tensor product, since invertibility can be checked on stalks.
  However, it is not always true that invertible objects are locally free of rank $1$.
  We will return to the discussion of this subtle point.
\end{rmk}

To visualize haunts, we can think of them as categorified versions of line bundles.
In particular, one may construct haunts in the same way that line bundles on a space $B$ are constructed, by picking an open cover of $B$ and gluing together trivial bundles with transition maps over intersections.
We will make this precise after introducing parametrized category perspective on haunts.

\subsection{PoV II: Parametrized Categories}

We now switch to the second point of view of haunts, by the following direct analogue of Theorem~\ref{thm:sheaves_vs_local_systems} in this situation.

\begin{theorem} \label{thm:Haunt_equiv_Map_to_BPicS}
  If $B$ is locally contractible, then we have
  \[
  \Haunt_{B,R} \simeq \Map(B,B\Pic(R)) \,.
  \]
  \begin{proof}
    Let $\mathcal{H}$ be a haunt.
    By assumption, we can find a cover $\{U_i\}$ of $B$ by contractible spaces.
    Since $U$ is contractible, we have $\Fun(U,\Sp) \simeq \Sp$.
    Hence, in this case any haunt is a locally constant sheaf of categories, taking the constant value $\Sp$.
    By Theorem ~\ref{thm:sheaves_vs_local_systems}, there is an equivalence between locally constant sheaves on $B$ valued in $\Pr^L_{\St}$ and local systems $\Fun(\Sing(B),\Pr^L_{\St})$.
    As our sheaf is constant with value $\Sp$, it takes values in the connected component of the core $\infty$-groupoid containing $\Sp$, and the local system necessarily factors though $B\Aut(\Sp) \simeq B\Pic(\bS)$.
  \end{proof}
\end{theorem}

\begin{rmk}
  If we want to consider invertible $\cO_{B,R}$-modules, rather than locally free rank 1 ones, the corresponding space will be equivalent to $\Map(B,\Br(R))$ instead.
  Keeping in mind that our main motivation comes from Floer homotopy theory, where the map corresponding to the relevant haunt always seems to factor through the real or complex J-homomorphism
  \[
  B(\bZ \times BO) \longto B\Pic(\bS) \quad \text{or} \quad B(\bZ \times BU) \longto B\Pic(\bS) \,,
  \]
  we will exclusively deal with the locally free rank $1$ case.
  It is possible that applications of twisted spectra to settings where we have to look at invertible modules, rather than locally free modules, exists within the mathematical wilderness, but as of the time of writing, the authors of the paper know of no such examples.
  As a potential application, we mention the introduction of~\cite{Szy17}, where the author muses, with reference to hints in~\cite{Cla12}, that the Brauer space may be the appropriate target for an elliptic J-homomorphism.
  Hence, it is possible that one needs to consider invertible modules in applications to elliptic cohomology.
\end{rmk}

Due to the above theorem, we will sometimes abusively refer to maps $B \to B\Pic(R)$ also as haunts, but we will try to stick to the terminology \emph{twists on the space $B$} if we want to make the distinction clear.
Given a haunt $\cH$, we will refer to its corresponding twist as $\tau_{\cH}$.
Given a twist $\tau : B \to B\Pic(R)$, we will refer to its corresponding haunt as $\cH_{(B,\tau)}$.
In particular, note that the identity map $1 : B\Pic(R) \to B\Pic(R)$ provides us with a haunt over the classifying space of the Picard space.
We  refer to this as the \emph{universal $R$-haunt} and denote it by $\mathcal{U}_R = \cH_{(B\Pic(R),1)}$.

\subsection{Haunts as categorified line bundles}
Let us discuss in detail how one may construct haunts ``by hand'', so that the reader may visualize these.
Let $B$ be a topological space and let $\{U_{i}\}$ denote an open cover of $B$.
We let
\[
p: \coprod_i U_i \to B
\]
denote the jointly surjective map and let $N_{\bullet}(p)$ be the nerve of this epimorphism.
This is a simplicial object in $\cS$ for which there is an equivalence
$$
\colim_{\Delta^{\oop}}N_{\bullet}(p) \simeq B
$$
where the topological space is now regarded as an object in $\cS$.

\begin{rmk} \label{cosimplicial shit}
We remark that $B \Pic(\bS)$ may itself be written as the geometric realization of the colimit diagram $B\Pic_{\bullet}$ which is nerve of the map $* \to B \Pic(\bS)$.
Thus, the map $B \to B \Pic(\bS)$ is equivalent to a map of colimit diagrams $N_{\bullet}(p) \to B \Pic_{\bullet}$.
\end{rmk}

Since the functor $\Map(-, B \Pic)$ sends colimits of spaces to limits, we may express a haunt $\cH_{(B, \tau)}$ corresponding to a twist $\tau : B \to B\Pic(\bS)$ as the totalization of an object in the cosimplicial space
\begin{equation} \label{cosimplicial_category}
   \Map(N_{\bullet}(p), B \Pic): = \Map(N_{0}(p), B \Pic) \rightrightarrows \Map(N_{1}(p), B \Pic) \substack{\rightarrow\\[-1em] \rightarrow \\[-1em] \rightarrow} \Map(N_{2}(p), B \Pic) \cdots \,.
\end{equation}
Composing level-wise with the inclusion $B\Pic(\bS) \to \on{Pr}^{L}_{\St}$, we may express the haunt $\cH_{(B, \tau)}$ itself as an object in the cosimplicial $\infty$-category
\[
\Fun(N_{\bullet}(p), \mathrm{Pr}^L_{\St}): = \left(\Fun(N_{0}(p), \mathrm{Pr}^L_{\St}) \rightrightarrows \Fun(N_{1}(p), \mathrm{Pr}^L_{\St}) \substack{\rightarrow\\[-1em] \rightarrow \\[-1em] \rightarrow} \Fun(N_{2}(p), \mathrm{Pr}^L_{\St}) \substack{\rightarrow \\[-1em] \rightarrow \\[-1em] \rightarrow \\[-1em]\rightarrow} \cdots \right) \,,
\]
Let us write
\[
\cH \simeq \lim_{\Delta} \cH_{(N_i(p),\tau_i)} \,.
\]
Note that since $B$ is assumed to be locally contractible, we can always choose the cover so that it trivializes the twist, so without loss of generality we will assume that we have such a cover.
In this case, each $\tau_i : N_i(p) \to B\Pic$ is nullhomotopic, and therefore gives rise to the trivial haunt on this restriction.
Thus, every $\cH_{N_{i}(p), \tau_i}$ will be equivalent to the trivial haunt $\cO_{(N_{i}(p))}$.
The cosimplicial structure abstractly determines equivalences on the restrictions of these haunts on $n$-fold intersections of the cover.
Unpacking what this means, and using Remark \ref{cosimplicial shit} above, we see that in simplicial degree $1$, we obtain maps
 \[
\coprod_{i,j} U_i \times_{B} U_j \to  \Pic(\bS) \,,
 \]
 encoding coherent choices of parametrized  invertible $\bS$-modules on intersections.
Similarly, in degree $2$, we obtain maps
\[
\coprod_{i,j,k} U_i \times_{B} U_j \times_{B} U_k \to  \Pic(\bS) \times \Pic(\bS) \,;
\]
a moment's reflection reveals that this is exactly the cocycle condition on triple intersections.

\subsection{Examples of haunts} \label{section_examples_haunts}

In this section, we will construct some explicit examples of haunts.
Recall that in order to construct examples of haunts on a space $B$, it is often enough to give gluing data on intersections $U_i \cap U_j$ of a trivializing open cover $\{U_i\}_{i \in I}$ in the form of a parametrized invertible spectrum.

\begin{example}[Haunts over the circle]
  Pick a open cover of the circle consisting of two open sets $U_0$ and $U_1$ with intersection consisting of two contractible connected components $V_0$ and $V_1$.
  The gluing data then amounts to an equivalence $v_i$ of the $\infty$-category of spectra on each connected component $V_i$.
  Let us pick
  \[
  v_0 = \mathrm{Id} : \Sp \longto \Sp \och v_1 = \Sigma^n : \Sp \longto \Sp \,.
  \]
  We refer to this haunt as the \emph{$n$th suspension haunt}.
  Since equivalence classes of haunts over $S^1$ are classified by homotopy classes of maps $S^1 \to B\Pic(\bS)$, that is, by $\pi_1 B\Pic(\bS) \cong \bZ$, there should be essentially one haunt for each integer $n$.
  Indeed, this is the $n$th suspension haunt that we constructed above.
\end{example}

\begin{example}[Haunts over the 2-sphere]
 We cover the 2-sphere by two open disks $U_0$ and $U_1$ containing the northern hemisphere and the southern hemisphere, respectively. The intersection is then an equatorial band which is homotopy equivalent to $S^1$, so we need to give an equivalence
 \[
 r_{01} : \Sp^{S^1} \longto \Sp^{S^1} \,.
 \]
 There is one non-trivial $S^0$-bundle over $S^1$, namely the real Hopf fibration
 \[
 S^0 \longto S^1 \longto S^1 \,.
 \]
 This gives us an invertible $S^1$-spectrum that we denote $\bS^\eta$.
 The equivalence $r_{01}$ is given by taking the smash product with with $\bS^\eta$.
 This is the only non-trivial haunt over $S^2$ there is since
 \[
 \pi_0 \Map(S^2,B\Pic(\bS)) = \pi_2 B\Pic(\bS) \cong \bZ/2 \,.
 \]
\end{example}

\begin{example}[Haunts over the 3-Sphere]
 We cover the 3-sphere by two open 3-disks $U_0$ and $U_1$.
 The intersection $U_0 \cap U_1$ is then homotopy equivalent to a 2-sphere.
 Note that there is one non-trivial $S^2$-bundle over $S^2$, namely the bundle
 \[
 S^2 \longto S^3 \times_{S^1} S^2 \longto S^2
 \]
 where $S^1$ acts on $S^2$ by rotation and on $S^3$ via the Hopf action.
 This bundle gives us an invertible $S^2$-spectrum $T$.
 The gluing equivalence
 \[
 r_{01} : \Sp^{S^2} \longto \Sp^{S^2}
 \]
 is given by taking the smash product with this spectrum.
 This is the only non-trivial haunt over $S^3$ there is since
 \[
 \pi_0 \Map(S^3,B\Pic(\bS)) = \pi_3 B\Pic(\bS) \cong \bZ/2 \,.
 \]
\end{example}

We will now give brief introduction of haunts, and phenomena related to haunts, that arise naturally in the setting of two different flavors of Floer theory: monopole Floer theory and symplectic Floer theory.

\begin{example}[Haunts in Seiberg--Witten Floer theory] \label{ex:SWF}
Let us give a brief sketch on how haunts are constructed from Seiberg--Witten Floer theory.
This is a condensed version of some of the results that will become available in the forthcoming article~\cite{BHK}.
Some of the basic statements on 3-manifolds equipped with complex spin structures can be found in~\cite{KM07}*{Chapter 1}.
One considers pairs $(Y,\mathfrak{s})$ where $Y$ is an oriented riemannian 3-manifold and $\mathfrak{s}$ is a complex spin structure on $Y$.
From this data one may canonically associate a class $[(Y,\mathfrak{s})] \in K^1(P)$ where
\[
P = H^1(Y;\bR)/H^1(Y;\bZ)
\]
is the so-called \emph{Picard torus} of $Y$ and $K^1(P)$ denotes complex topological K-theory of the Picard torus in cohomological degree $1$.
There are a couple of different ways of constructing this class, but to keep it simple let us give a construction of the class completely contained within complex K-theory.
One auxiliary bundle which is convenient to use is the so-called \emph{Poincaré bundle} $\mathcal{L} \to P \times Y$.
This is the bundle whose total space is given as
\[
\mathcal{L} = P \times \widetilde{Y} \times \bC / \pi_1(Y)
\]
where the second term denotes the universal cover of $Y$, and where the fundamental group acts as
\[
\gamma \cdot (a,x,z) = (a,\gamma \cdot x , \exp(ia(\gamma))z) \,.
\]
Here, the notation in the second factor just refers to the ordinary action the fundamental group on the universal cover and the notation $a(\gamma)$ in the third factor comes from remembering that $a \in H^1(Y;\bR)/H^1(Y;\bZ)$ and that $H^1(Y;\bR) \cong \Hom(\pi_1(Y),\bR)$.
We remark that the Poincaré bundle should be viewed as a family of flat line bundles over $Y$ parametrized by $P$. Indeed, the following statements can be used to characterize it:
\begin{itemize}
  \item From the point of view of the Riemann--Hilbert correspondence these line bundles with flat connection are coming from the family of unitary representations
  \[
  \rho_a : \pi_1(Y) \longto U(1) \,, \quad \gamma \mapsto \exp(ia(\gamma)) \,.
  \]
  \item One can check that the first Chern character of the Poincaré bundle is precisely the element in $H^2(P \times Y ;\bZ)$ corresponding to the identity morphism on $\pi_1(P)$ under the isomorphism
  \[
  \Hom(\pi_1(P),\pi_1(P)) \cong H^1(P;\bZ) \otimes H^1(Y;\bZ) \,.
  \]
\end{itemize}
With this bundle in the back of our head, let us construct the wanted K-theory class.
The data of the complex spin structure $\mathfrak{s}$ can alternatively be phrased in terms of a hermitian rank $2$ bundle $S \to Y$, the spinor bundle, which comes equipped with Clifford multiplication.
Note that we can consider the isomorphism class $[S]$ of the spinor bundle as an object of $K^0(Y)$ and that we can consider the isomorphism class $[\mathcal{L}]$ of the Poincaré bundle as an object of $K^0(P \times Y)$.
Let us consider the projections $p : P \times Y \longto P$ and $q : P \times Y \longto Y$. Note that we have a pullback map
\[
q^* : K^0(Y) \longto K^0(P \times Y) \,.
\]
Moreover, since $Y$ is equipped with a complex spin structure, which is equivalent to an orientation with respect to complex topological K-theory by~\cite{ABS64}, we are provided with a umkehr map
\[
p_! : K^0(P \times Y) \longto K^{-3}(P) \cong K^1(P) \,.
\]
We define
\[
[(Y,\mathfrak{s})] = p_!(q^*([S]) \otimes [\mathcal{L}]) \in K^1(P) = [P,U] \,.
\]
We note that the same class in K-theory can be obtained through other methods, like for example applying index theory to the family of Dirac operators associated to $(Y,\mathfrak{s})$.
To get a well-defined haunt on the Picard torus, we use Bott periodicity and the complex J-homomorphism to end up with a twist
\[
\tau_{(Y,\mathfrak{s})} : P \longto U \simeq B(\bZ \times BU) \longto B\Pic(\bS) \,.
\]
In particular, this map is not trivializable when $Y$ is the 3-torus, and one needs to take this into account when constructing stable homotopy types in that situation.
\end{example}

\begin{example}[Haunts in Lagrangian Floer theory]
  Given a lagrangian submanifold $L$ in some symplectic manifold we can construct the so-called \emph{stable Gauss map}
  \[
  \gamma : L \longto U/O \,.
  \]
  This gives rise to a twist on $L$ via real Bott periodicity and the real J-homomorphism
  \[
  \tau : L \longto U/O \simeq B(\bZ \times BO) \longto B\Pic(\bS)\,.
  \]
  We remark that for a commutative ring $k$, the composition
  \[
    \tau_{Hk} : L \longto B\Pic(\bS) \longto B\Pic(Hk)
  \]
  picks up the Maslov class on $\pi_1$ and the second relative Stiefel--Whitney class on $\pi_2$.
  The possible non-triviality of this map explains why certain Floer homology theories are not integrally graded.
  It is known that the map $\tau_{Hk}$ is zero on homotopy groups when dealing with exact lagrangians in cotangent bundles; see for example~\cite{AK16} and~\cite{Jin20} for proofs using two different methods.
\end{example}

\section{Twisted Spectra}

In this section, we introduce the focal notion of study in this paper, namely twisted spectra.
Again, twisted spectra are global sections of haunts, and since we have two main perspectives on haunts, this means that \emph{a priori} we also have two perspectives on twisted spectra.
Moreover, it turns out that there is a third perspective: twisted spectra are simply modules over Thom spectra.
While this is certainly a fruitful perspective, much is known about modules categories and Thom spectra, viewing twisted spectra solely from this perspective tends to obscure some important structural insights on how twisted spectra behave.
In this section, we will discuss such structural results, such as symmetric monoidal structures, functoriality with respect to maps of spaces, and more.

\subsection{Twisted Spectra}

We now define the main objects of this paper.

\begin{definition}
  A global section of an $R$-haunt is called a \emph{twisted $R$-module}. If $R = \bS$, we will simply call these objects \emph{twisted spectra}.
\end{definition}

In the original paper by Douglas, twisted spectra were referred to as ``spectres'', which also explains the terminology ``haunt''.
We will use both terminologies, but stick to twisted spectra for most of the time.
The category of twisted $R$-modules associated to the $R$-haunt $\mathcal{H}$ is the global sections of the haunt.
We will denote this category by
\[
\mathrm{TwMod_R}(\cH) = \Gamma \cH \,.
\]
Recall that taking global sections of a haunt over a space $B$ corresponds to taking the limit of the corresponding parametrized category.
Hence, we may alternatively define the $\infty$-category of twisted spectra corresponding to a haunt $\mathcal{H}$ over $B$ as
$$
\TwSp(B,\tau_{\cH}) := \lim_{B}(B \overset{\tau_{\cH}}\longto B\Pic(\bS) \longto \mathrm{Pr}^L_{\St}) \,.
$$
Similarly, the $\infty$-category of twisted $R$-modules corresponding to an $R$-haunt $\tau : B \to B\Pic(R) $ is the limit
\[
\mathrm{TwMod}(B,\tau) = \lim_{B}(B \overset{\tau}\longto B\Pic(R) \longto \Cat_R) \,.
\]
We will typically adopt this perspective for the remainder of this work.
Note also, and this will be important, that ambidexterity tells us that we might as well have replaced the limit by the colimit in the above formula.
Many results on twisted spectra arises from interpolating between thinking of $\TwSp(B,\tau)$ as a colimit and as a limit.

\begin{rmk}[Visualizing twisted spectra]
Let $B$ be a space with a twist $\tau: B \to B \Pic(\bS)$.
Following the discussion of Section ~\ref{section_examples_haunts}, we pull back the twist $\tau$ along a trivializing open cover $p: \coprod_{i \in I} U_i \to B$ to obtain the following limit diagram in $\mathrm{Pr}^{L}_{\St}$:
\begin{equation} \label{big_kahuna}
\TwSp(B, \tau) \to \left(\Fun(N_{0}(p), \Sp) \rightrightarrows \Fun(N_{1}(p), \Sp) \substack{\rightarrow\\[-1em] \rightarrow \\[-1em] \rightarrow} \Fun(N_{2}(p), \Sp) \cdots \right) \,.
\end{equation}
Thus, to specify an object in $\TwSp(B, \tau)$, it is enough to specify compatible choices of parametrized spectra on the open cover.
Hence, a twisted spectrum $X$ determines (untwisted) parametrized spectra $X_i \in \Fun(U_i, \Sp)$, together with equivalences
\begin{equation}
X_i|_{U_i \cap U_j} \simeq  \Phi_{ij}(X_j|_{U_i \cap U_j})
\end{equation}
in $\Fun(U_i \cap U_j, \Sp)$, where $\Phi_{ij}$ is an auto-equivalence of $\Fun(U_i \cap U_j, \Sp)$.
Furthermore, these auto-equivalences are determined by invertible spectra satisfying a cocycle condition, so that in fact
$$
X_i|_{U_i \cap U_j} \simeq  \mathcal{L}_{ij} \otimes X_j|_{U_i \cap U_j}
$$
for $\mathcal{L}_{ij} \in \Map(U_i \cap U_j, \Pic(\bS))$.
\end{rmk}

\begin{rmk}
We expect that twisted spectra should be related to the notion of twisted vector bundles in a natural way.
Roughly, a twisted vector bundle on a space is obtained by gluing vector bundles $\{V_i\}$ along on open cover, $\{ U_i\}_{i \in I}$ where the standard equivalence on intersections is replaced by an equivalence
\[
V_i|_{U_i \cap U_j} \simeq V_{j}|_{U_i \cap U_j}  \otimes \mathcal{L}_{ij} \,.
\]
Here $\mathcal{L}_{ij}$ is an auxiliary line bundle; the collection of line bundles on intersections should satisfy a cocycle condition. See~\cite[Section 1]{freedhopkinsteleman} for a more precise discussion.
We will see that the descent data needed to glue together a twisted spectrum from parametrized spectra is exactly of this form. This leads us to speculate about the existence of a ``twisted J-homomorphism''.
Roughly, fixing the space $B$, we expect this to be a functor
\[
\on{TwVect}(B) \to \Pic(\TwSp(B)),
\]
where $\on{TwVect}(B)$ is a suitably chosen category of twisted vector bundles on $B$, and $\TwSp(B)$ is a total category of twisted spectra which we introduce in Section~\ref{sec:total_category}.
We will not, however, pursue this line of investigation any further herein.
\end{rmk}

\subsection{Examples of Twisted Spectra}

Before proving structural results about twisted spectra, let us give some explicit examples to guide the reader's intuition.
The reader should keep in mind the various examples of haunts that we gave in Section~\ref{section_examples_haunts}.
We will typically make reference to and also use the same notation used in those examples.

\begin{example}[Twisted Spectra over the Circle]
  By the above discussion, we have that
  \[
  \TwSp(S^1, n) \overset{\simeq}\to \mathrm{eq} \left(\Fun(U_0 , \Sp) \times \Fun(U_1,\Sp) \rightrightarrows \Fun(V_0,\Sp) \times \Fun(V_1,\Sp) \right)
  \]
  for the $\infty$-category of twisted spectra corresponding to the $n$th suspension haunt.
  In particular, we may think of $\TwSp(S^1, n)$ as sitting in a pullback diagram
  \[
  \begin{tikzcd}
    \TwSp(S^1,n) \arrow[r] \arrow[d] & \Sp \arrow[d] \\
    \Sp \arrow[r] & \Sp \times \Sp
  \end{tikzcd}
  \]
  where the bottom horizontal map sends a spectrum $X$ to $(X,X)$ and the right-hand vertical map sends a spectrum $Y$ to $(Y,\Sigma^n Y)$.
  Hence, a twisted spectrum for the $n$th suspension haunt is then essentially a spectrum $X$ together with an equivalence $\phi : \Sigma^n X \simeq X$.
\end{example}

\begin{example}[Twisted Spectra over the 2-Sphere]
  There is only one non-trivial haunt over $S^2$, and by the above, its category of twisted spectra can be described as
  \[
  \TwSp(S^2,\tau) \simeq \mathrm{eq} \left(\Fun(U_0 , \Sp) \times \Fun(U_1,\Sp) \rightrightarrows \Fun(U_0 \cap U_1 ,\Sp) \right) \,.
  \]
  In particular, we may think of $\TwSp(S^2, \tau)$ as sitting in a pullback diagram
\[
\begin{tikzcd}
\TwSp(S^2, \tau) \arrow[r] \arrow[d] & \Sp \arrow[d,"X \mapsto \underline{X} \otimes S^{\eta}"] \\
\Sp \arrow[r, "Y \mapsto \underline{Y}"] & \Sp^{S^1} \,.
\end{tikzcd}
\]
where the bottom map takes a spectrum $Y$ and constructs the trivially parametrized $S^1$-spectrum $\underline{Y}$, and the right vertical map sends $X$ to $\underline{X} \otimes S^\eta$.
Recall now the equivalence
\[
\Sp^{S^1} \simeq \on{RMod}_{\bS[\bZ]} \,;
\]
the latter $\infty$-category can be thought of as a spectrum equipped with an automorphism.
Viewed in this way, we can reinterpret the bottom map as the one assigning the identity $\id_{Y}$ to the spectrum $Y$, and the right vertical map as assigning the automorphism $-\id_{X}$ to $X$.
From this description, we may think of an object of $\TwSp(S^2, \tau)$ as a pair $(X,\iota)$ where $X$ is a spectrum and where $\iota$ is a homotopy $\iota: \id_{X} \simeq -\id_{X}$.
\end{example}

\begin{example}[Twisted Spectra over the 3-Sphere]
  There is only one non-trivial haunt over $S^3$, and by the above, its category of twisted spectra can be described as
  \[
  \TwSp(S^3,\tau) \simeq \mathrm{eq} \left(\Sp \rightrightarrows \Fun(S^2,\Sp) \right) \,.
  \]
  In particular, we can think of $\TwSp(S^3,\tau)$ as sitting in the pullback diagram
  \[
  \begin{tikzcd}
    \TwSp(S^3,\tau) \arrow[r] \arrow[d] & \Sp \arrow[d] \\
    \Sp \arrow[r] & \Sp^{S^2}
  \end{tikzcd}
  \]
  where the bottom map takes a spectrum $X$ and constructs the trivially parametrized $S^2$-spectrum $\underline{X}$ and the second map takes a spectrum $Y$ and sends it to $\underline{Y} \otimes T$. A twisted spectrum over the 3-sphere is then a spectrum~$X$ together with an equivalence $\underline{X} \simeq \underline{X} \otimes T$.
  Let us analyze this a bit further.
  Recall that we have an equivalence
  \[
  \Sp^{S^2} \simeq \on{RMod}_{\bS[\Omega S^2]}\\,,
  \]
  of $\infty$-categories.
  The map $\eta: S^1 \to \Aut(\bS) \simeq \on{GL}_1(\bS)$ classifying $\eta \in \pi_1(\on{GL}_1(\bS))$, factors  as a map
    \[
  S^1 \to \Omega S^2 \to \Aut(\bS) \simeq \on{GL}_1(\bS),
  \]
  where $\Omega S^2 \simeq \Omega \Sigma S^1$ is the free $\bE_1$-algebra on $S^1$, since $\on{GL}_1(\bS)$ is an $E_1$ algebra.
  Now, given a spectrum $X$, we may view it as a $\bS[\Omega S^2]$-module via restriction of scalars along $\bS[\Omega S^2] \to \bS$.
  Since an automorphism of $\bS$ induces an automorphism of $\bS \otimes X \simeq X$,  this induces a map
  \[
 \eta_X: S^1 \to \Omega S^2 \to \Aut(X) \simeq \on{GL}_1(X),
  \]
for a general spectrum $X$.
The equivalence $\underline{X} \simeq \underline{X} \otimes T$ precisely corresponds to a null homotopy of $\eta_X$.
Hence, we can think of an object in $\TwSp(S^3,\tau)$ as a pair $(X,\alpha)$ where $X$ is a spectrum and $\alpha$ is a null homotopy of $\eta_X$.
\end{example}

\begin{example}[Twisted Spectra in Seiberg--Witten Floer theory] \label{ex:SWF_twisted_spectrum}
  Let us discuss briefly how twisted spectra are constructed from Seiberg--Witten Floer data.
  We already explained how the twist was constructed completely within K-theory.
  This is perhaps not the best way to construct the twist if one wants to build a spectrum-like object from geometric data, though.
  In practice, one builds an explicit cocycle description of the class by studying a ``total Dirac operator''  on a certain Hilbert bundle over the corresponding Picard torus:
  \[
  \begin{tikzcd}
    \mathcal{H} \arrow[rr,"\mathbb{D}"] \arrow[rd] & & \mathcal{H} \arrow[ld] \\
    & P \,.&
  \end{tikzcd}
  \]
  One then studies the distribution of eigenvalues of the operator $\mathbb{D}$ and picks an open cover $\{U_i\}$ of $P$ in a suitable way with respect to this eigenvalue distribution.
  Once such an open cover is chosen, one constructs the spectra~$X_i$ on $U_i$ following the constructions of~\cite{Man03},~\cite{KLS}, and~\cite{SS}, using finite-dimensional approximation and Conley indices.
  See also~\cite{Kha13}.
  The resulting twisted spectrum associated to the pair $(Y,\mathfrak{s})$ will be denoted $\mathrm{SWF}(Y,\mathfrak{s})$.
  Again, more details will appear in the forthcoming article~\cite{BHK}.
\end{example}

\subsection{Pullbacks along maps of spaces} \label{sec:pullback}

We now move to some important functoriality of the categories of twisted spectra with respect to maps of spaces.
We first assert that we have the standard pullback, and its left and right adjoints.

\begin{proposition} \label{prop:functoriality_twisted_spectra}
Given a map $f : A \to B$ of spaces we have a pullback functor
\[
f^* : \TwSp(B,\sigma) \longto \TwSp(A,f^* \sigma) \,.
\]
This functor has left and right adjoints that we will denote
\[
f_! : \TwSp(A,f^* \sigma) \longto \TwSp(B,\sigma) \och f_* : \TwSp(A,f^* \sigma) \longto \TwSp(B,\sigma) \,,
\]
respectively.
\end{proposition}
\begin{proof}
  Follows from applying Proposition~\ref{prop:left_right_adjoints_ambidexterity} to the composite functor $B \to B\Pic(\bS) \to \Pr^{L}_{\St}$.
\end{proof}

Note that the existence of the left adjoint $f_!$ crucially uses ambidexterity, while the right adjoint $f_*$ more or less follows from our construction of $f^*$ as a left adjoint.
Note that the above provide us with functors
\begin{equation*}
  \begin{aligned}
\TwSp_{!} &: \cS_{/B\Pic(\bS)} \longto \mathrm{Pr}^{L}_{\St} \,, \quad &f \mapsto f_!  \\
\TwSp^* &: (\cS_{/B\Pic(\bS)})^{\oop} \longto \mathrm{Pr}^{L,R}_{\St} \,, \quad &f \mapsto f^* \\
\TwSp_* &: \cS_{/B\Pic(\bS)} \longto \mathrm{Pr}^{R}_{\St} \,, \quad &f \mapsto f_*
\end{aligned}
\end{equation*}
all of which have the same assignment on objects and where the assignment on a morphism $f : (A,f^*\sigma) \to (B,\sigma)$ is shown on the right.
Later on, we will want to show that the functors described above have the expected compatibility with the tensor product of twisted spectra, which will be defined in the next subsection.
An important result that will help with this is the Beck--Chevalley condition, which tells us how the functors above behave with respect to change of base.

\begin{proposition}[Beck--Chevalley] \label{prop:Beck_Chevalley_preliminary}
  Suppose that
  \[
  \begin{tikzcd}
    (A,\tau) \arrow[r,"j"] \arrow[d,"i"'] & (B,\sigma) \arrow[d,"f"] \\ (C,\rho) \arrow[r,"g"] & (D,\pi)
  \end{tikzcd}
  \]
  is a pullback square in $\cS/_{B\Pic(\bS)}$. Then the canonical transformation
  \[
  i_! j^* \longto g^* f_!
  \]
  is an equivalence of functors $\TwSp(B,\sigma) \to \TwSp(C,\rho)$.
  \begin{proof}
    In order to prove the result, we temporarily adopt the perspective of twisted spectra as the global sections of a sheaf of categories.
    Let us pick a trivializing open cover for the haunt $\cH_{(D, \pi)}$.
    By pullback along $g$ and $h$, this will trivialize the haunts $\cH_{(C, \rho)}$ and $\cH_{(B, \sigma)}$ respectively.
    Taking the {\v C}ech nerve of the covers, and using the description  of $\TwSp(-,-)$ as the totalization of the cosimplicial object in categories (cf. (~\ref{big_kahuna})), with each term being of the form $\Fun(U_i, \Sp)$, we can re-express the  natural transformation $i_! j^* \to g^* f_!$ as a natural transformation of functors
 $$
 \Fun( U_\bullet ,\Sp) \to \Fun(V_\bullet, \Sp)
 $$
 where $U_\bullet$ is the {\v C}ech nerve of the cover of $B$ and $V_\bullet$ is the {\v C}ech nerve of the cover of $C$.
 Note that this decomposes, in each cosimplicial degree $i$, to the ordinary Beck-Chevalley maps
$$
i_!j^* \to g^* f_!: \Fun(V_i, \Sp) \to \Fun(U_i, \Sp),
$$
which are well known to be equivalences.
Thus, we have displayed the natural transformation in question as  natural transformation of functors of cosimplicial objects in categories, with each component being an equivalence.
It follows that the natural transformation on totalizations will itself be an equivalence, as we wanted to show.
  \end{proof}
\end{proposition}

\subsection{Tensor products of twisted spectra} \label{sec:tensor_prod_TwSp}

Let us start by considering some monoidal structures on the $\infty$-category of spaces over $B\Pic(\bS)$.
To simplify terminology, we will refer to the overcategory $\cS_{/B\Pic(\bS)}$ as the $\infty$-category of \emph{twisted spaces}.
Note that this $\infty$-category is symmetric monoidal by intertwining the cartesian product on $\cS$ with the $\bE_\infty$-structure on $B\Pic(\bS)$.
More specifically, the symmetric monoidal structure, which will be denoted $\boxtimes$ and referred to as the \emph{exterior product}, is given as follows: if $(A,\tau)$ and $(B,\sigma)$ are twisted spaces, we have
\[
(A, \tau) \boxtimes (B,\sigma) = (A \times B , \tau \boxtimes \sigma)
\]
where $\tau \boxtimes \sigma$ denotes the composition
\[
\tau \boxtimes \sigma : A \times B \overset{\tau \times \sigma}\longto B\Pic(\bS) \times B\Pic(\bS) \longto B\Pic(\bS) \,.
\]
Compare this to the \emph{interior product} on the mapping space $\Map(A,B\Pic(\bS))$ for a fixed space $A$ which is given by using the diagonal on $A$ together with the $\bE_\infty$-structure on the Picard space:
\[
 +  : \Map(A,B\Pic(\bS)) \times \Map(A,B\Pic(\bS)) \longto \Map(A \times A, B\Pic(\bS) \times B\Pic(\bS)) \longto \Map(A,B\Pic(\bS))
\]
In particular, note that the interior product $\tau + \sigma$ is obtained by pulling back the exterior product $\tau \boxtimes \sigma$ along the diagonal map $\Delta_A : A \to A \times A$.
Similar to the above situation, twisted spectra will also have an exterior and an interior product.
We obtain the first one by considering tensor products of categories of twisted spectra.

\begin{proposition}\label{prop:tensor_twisted_categories}
  The functor $\TwSp_! : \cS_{/B\Pic(\bS)} \to \Pr^{L}_{\St}$ is symmetric monoidal with respect to the exterior product on the source and the Lurie tensor product on the target.
  In particular, we have
  \[
  \TwSp(A,\tau) \otimes \TwSp(B,\sigma) \simeq \TwSp(A \times B , \tau \boxtimes \sigma) \,,
  \]
  and the diagram
  \[
  \begin{tikzcd}
    \TwSp(A, \tau) \otimes \TwSp(B,\sigma) \arrow[d,"f_! \otimes g_!"] \arrow[r,"\simeq"] & \TwSp(A \times B, \tau \boxtimes \sigma) \arrow[d,"(f \times g)_!"] \\
    \TwSp(C,\rho) \otimes \TwSp(D,\pi) \arrow[r,"\simeq"] & \TwSp(C \times D , \rho \boxtimes \pi)
  \end{tikzcd}
  \]
  commutes.
  \begin{proof}
    Note that the functor $\TwSp$ that we are referring to in the proposition is the composition
    \[
    \cS_{/B\Pic(\bS)} \longto \cS_{/\Pr^{L}_{\St}} \overset{\colim}\longto \mathrm{Pr}^{L}_{\St} \,.
    \]
    Note that these are all symmetric monoidal functors.
    That the diagram is commutative follows from the fact that the coherence maps of a symmetric monoidal functor is supposed to be natural.
  \end{proof}
\end{proposition}

\begin{corollary} \label{cor:natural_pullback_and_right_adjoint}
  Given maps $f : A \to C$ and $g : B \to D$ of spaces. The following holds:
  \begin{enumerate}
    \item The diagram
    \[
    \begin{tikzcd}
      \TwSp(C,\tau) \otimes \TwSp(D,\sigma) \arrow[d,"f^* \otimes g^*"] \arrow[r,"\simeq"] & \TwSp(C \times D , \tau \boxtimes \sigma) \arrow[d,"(f \times g)^*"] \\
      \TwSp(A, f^* \tau) \otimes \TwSp(B,f^* \sigma) \arrow[r,"\simeq"] & \TwSp(A \times B, f^* (\tau \boxtimes \sigma))
    \end{tikzcd}
    \]
    commutes.
    \item the diagram
    \[
    \begin{tikzcd}
      \TwSp(A, f^* \tau) \otimes \TwSp(B,f^* \sigma) \arrow[d,"f_\ast \otimes g_\ast"] \arrow[r,"\simeq"] & \TwSp(A \times B, f^* (\tau \boxtimes \sigma)) \arrow[d,"(f \times g)_\ast"] \\
      \TwSp(C,\tau) \otimes \TwSp(D,\sigma) \arrow[r,"\simeq"] & \TwSp(C \times D , \tau \boxtimes \sigma)
    \end{tikzcd}
    \]
    commutes.
  \end{enumerate}
  \begin{proof}
    The first statement follows from the functors $f^*$, $g^*$, and $(f \times g)^*$ being right adjoints to $f_!$, $g_!$, and $(f \times g)_!$, respectively. The second statement then follows from the first statement since $f_*$, $g_*$, and $(f \times g)_*$ are right adjoints to $f^*$, $g^*$, and $(f \times g)^*$, respectively.
  \end{proof}
\end{corollary}

We can use this equivalence to define the exterior tensor product of twisted spectra, simply by appealing to the universal property of tensor products of presentable stable $\infty$-categories.

\begin{definition}
  The \emph{exterior product of twisted spectra} is the functor
  \[
  - \boxtimes - : \TwSp(A, \tau) \times \TwSp(B,\sigma) \longto \TwSp(A \times B , \tau \boxtimes \sigma)
  \]
  corresponding to the equivalence of Proposition~\ref{prop:tensor_twisted_categories} under the universal property of tensor products of presentable stable $\infty$-categories.
\end{definition}

Note that the diagonal map on a space $A$ induces a pullback functor
\[
\Delta^*_A : \TwSp(A \times A , \tau \boxtimes \sigma) \longto \TwSp(A,\tau + \sigma)
\]
which can be used to define the interior product of twisted spectra.

\begin{definition} \label{defn:interior_tensor}
  The \emph{interior product of twisted spectra} is the functor
  \[
  - \otimes - : \TwSp(A,\tau) \times \TwSp(A,\sigma) \longto \TwSp(A , \tau + \sigma)
  \]
  obtained by pulling the exterior product back along the diagonal map on the space $A$.
\end{definition}

The following lemma alternatively describes the external product in terms of the internal product.

\begin{lemma} \label{lem:boxtimes_proj_pullbacks}
  Let $X \in \TwSp(A,\tau)$ and let $Y \in \TwSp(B,\sigma)$. Then
  \[
  X \boxtimes Y \simeq p_A^* X \otimes p_B^* Y
  \]
  where $p_A :  A\times B \to A$ and $p_B : A \times B \to B$ denotes the projections. Here, the tensor product is taken in twisted spectra over $A \times B$.
  \begin{proof}
    We have
    \begin{align*}
    p^*_A X \otimes p_B^* Y &= \Delta^{*}_{A \times B}(p_A^* X \boxtimes p_B^*) \\ &\simeq \Delta_{A \times B}^*((p_A \times p_B)^*(X \boxtimes Y)) \\ &\simeq X \boxtimes Y \,,
  \end{align*}
  where we are using Proposition~\ref{prop:boxtimes_pull}; see below.
  \end{proof}
\end{lemma}

If we want to be explicit about the fact that we work with twisted spectra over $A$, we will sometimes denote the interior product by $\otimes_A$, but in most cases, we will be implicit about this and simply write $\otimes$.
Note that the fact that the twists add up under the interior tensor product suggests that there is a symmetric monoidal structure on some category of twisted spectra over $A$ where we let the twists vary.
We will look at this category is Section~\ref{sec:total_category}.

\subsection{Projection isomorphism}

In this section, we will prove a preliminary version of the projection isomorphism for twisted spectra.
We begin with some compatibilities between pullbacks and the symmetric monoidal structure on twisted spectra, giving us a rudimentary version of pullbacks being symmetric monoidal with respect to the interior tensor product of twisted spectra.

\begin{proposition} \label{prop:boxtimes_pull}
  Let $f : A \to C$ and $g : B \to D$. Let $X \in \TwSp(C,\rho)$ and $Y \in \TwSp(D,\pi)$. Then we have
  \[
  (f \times g)^*(X \boxtimes Y) \simeq f^* X \boxtimes g^* Y \,.
  \]
  \begin{proof}
  The statement of the theorem essentially follows from the commutativity of the diagram
    \[
    \begin{tikzcd}
      \TwSp(C, \rho) \otimes \TwSp(D,\pi) \arrow[d,"f^* \otimes g^*"] \arrow[r,"\simeq"] & \TwSp(C \times D, \rho \boxtimes \pi) \arrow[d,"(f \times g)^*"] \\
      \TwSp(A,f^* \rho) \otimes \TwSp(B,g^* \pi) \arrow[r,"\simeq"] & \TwSp(A \times B , f^* \rho \boxtimes g^* \pi) \,,
    \end{tikzcd}
    \]
    in Proposition~\ref{prop:tensor_twisted_categories}.
  \end{proof}
\end{proposition}

\begin{corollary} \label{cor:pullback_symm_mon_prelim}
  Let $f : A \to B$ be a map of spaces, $X \in \TwSp(B,\tau)$ and $Y \in \TwSp(B,\sigma)$. Then we have
  \[
  f^*(X \otimes Y) \simeq f^* X \otimes f^* Y \,.
  \]
  \begin{proof}
    This follows from the computation
    \begin{align*}
      f^*(X \otimes Y) &\simeq f^*((\Delta_B)^*(X \boxtimes Y)) \\ &\simeq (\Delta_A)^*((f \times f)^*(X \boxtimes Y)) \\ &\simeq f^*X \otimes f^* Y \,.
    \end{align*}
  \end{proof}
\end{corollary}

\begin{proposition} \label{prop:proj_iso_prelim}
  Let $f : A \to B$ be a map of spaces, $Y \in \TwSp(B,\sigma)$, and $X \in \TwSp(A,f^* \rho)$ for some twists $\sigma , \rho : B \to B \Pic(\bS)$. Then we have an equivalence
  \[
  f_!(f^* Y \otimes X) \simeq Y \otimes f_! X
  \]
  of $(A,\sigma + \rho)$-twisted spectra.
  \begin{proof}
    Using the definition of the internal tensor product, the first part of Corollary~\ref{cor:natural_pullback_and_right_adjoint}, and commutativity of some maps we get:
    \begin{align*}
      f_!(f^* \otimes X) &\simeq f_! (\Delta_A^*(f^* Y \boxtimes X)) \\
      &\simeq f_!(\Delta_A^*((f \times 1)^*(Y \boxtimes X))) \\
      &\simeq f_!( ((f \times 1) \circ \Delta_A)^*(Y \boxtimes X)) \\
      &\simeq f_!((f,1)^*(Y \boxtimes X)) \,.
    \end{align*}
    We finish the proof by using Proposition~\ref{prop:Beck_Chevalley_preliminary} on the pullback diagram
    \[
    \begin{tikzcd}
      (A,f^* \sigma + f^* \rho) \arrow[d,"f"'] \arrow[r,"{(f,1)}"] & (B \times A,\sigma \boxtimes f^* \rho) \arrow[d,"1 \times f"] \\
      (B,\sigma + \rho) \arrow[r,"\Delta_B"] & (B \times B, \sigma \boxtimes \rho) \,,
    \end{tikzcd}
    \]
    so that we have equivalences
    \[
    f_!((f,1)^*(Y \boxtimes X)) \simeq \Delta_B^*((1 \times f)_!(Y \boxtimes X)) \simeq \Delta_B^*(Y \boxtimes f_! X) \simeq Y \otimes f_! X \,.
    \]
  \end{proof}
\end{proposition}

\subsection{Closed structure} \label{sec:closed_individualtwists}

Twisted spectra also have a closed structure.
Indeed, the exterior and internal tensor product on twisted spectra admit right adjoints, which we will refer to as the exterior and internal hom-twisted spectra.
Let us start with the exterior tensor product.
Let $Y$ be a $(B,\sigma)$-twisted spectrum and consider the functor $- \boxtimes Y : \TwSp(A,\tau) \to \TwSp(A \times B, \tau \boxtimes \sigma)$. We claim that this has a right adjoint that we will denote as
\[
\hom^{\boxtimes}(Y,-) : \TwSp(A \times B , \tau \boxtimes \sigma) \longto \TwSp(A,\tau)
\]
This follows formally from the fact that $- \boxtimes Y$ preserves colimits, and that we are dealing with presentable stable $\infty$-categories.

\begin{definition}
  The \emph{external hom object of twisted spectra} is the functor
  \[
  \hom^{\boxtimes}(-,-) : \TwSp(B,\sigma)^{\oop} \times \TwSp(A \times B , \tau \boxtimes \sigma) \longto \TwSp(A,\tau) \,.
  \]
\end{definition}

Moreover, we have a right adjoint to the functor $- \otimes Y : \TwSp(B,\tau) \to \TwSp(B,\tau + \sigma)$ which we denote by
\[
 \hom(Y,-) : \TwSp(A,\tau + \sigma) \longto \TwSp(A , \tau) \,.
\]
Since the interior product is obtained by pulling back the exterior product along the diagonal map, it follows that the internal hom twisted spectra functor is the composition
\[
\TwSp(A,\tau + \sigma) \overset{(\Delta_A)_*} \longto \TwSp(A \times A , \tau \boxtimes \sigma) \overset{\hom^{\boxtimes}(Y,-)}\longto \TwSp(A,\tau) \,.
\]

\begin{definition}
  The internal hom object of twisted spectra is the functor
  \[
  \hom(-,-) : \TwSp(A,\sigma)^{\oop} \times \TwSp(A , \tau + \sigma) \longto \TwSp(A,\tau)
  \]
\end{definition}

Let us prove a preliminary version of pullbacks being closed monoidal with respect to the internal hom objects.

\begin{proposition}
  Let $f : A \to B$ be a map of spaces, $X \in \TwSp(B,\tau)$, and $Y \in \TwSp(B,\sigma)$. Then we have an equivalence
  \[
  f^* \hom(X,Y) \simeq \hom(f^* X, f^* Y) \,
  \]
  of $(A,-f^* \tau + f^* \sigma)$-twisted spectra.
  \begin{proof}
    Let $Z$ be an arbitrary $(A,-f^* \tau + f^* \sigma)$-twisted spectrum.
    By using Corollary~\ref{cor:pullback_symm_mon_prelim} and Proposition~\ref{prop:proj_iso_prelim} as well as adjunctions, we obtain
    \begin{align*}
      \Map(Z,f^* \hom(X,Y)) &\simeq \Map(f_! Z , \hom(X,Y)) \\&\simeq \Map(f_! Z \otimes X , Y) \\ &\simeq \Map(f_!(Z \otimes f^* X),Y) \\ &\simeq \Map(Z \otimes f^* X , f^* Z) \\ &\simeq \Map(Z, \hom(f^* X , f^* Z))
    \end{align*}
    from which the result follows by a Yoneda lemma argument.
  \end{proof}
\end{proposition}

\subsection{Thom Spectrum Perspective} \label{sec:TwSp_ModTh}

Before we move on, let us also explain to the reader how to think about the formalism that we have just covered from the perspective of thinking of twisted spectra as modules over Thom spectra.
Doing this, we will follow some of the results in~\cite{CCRY23}, and extend them slightly to account for functoriality in spaces.
Let us consider the adjunction
\[
\Mod_{(-)} : \Alg(R) \leftrightarrows (\Cat_{R})_{\Mod_R/} : \End_R(1_{(-)}) \,.
\]
If $A$ is a connected space, let us pick a base point $a$ and consider the inclusion $a \to A$, which we by abuse of notation will simply denote $a$.
If $\tau : A \to B\Pic$ is a twist, then note that we have an induced functor $a^{\tau}_! : \Sp \to \TwSp(A,\tau)$ which makes $\TwSp(A,\tau)$ an object of $(\Cat_{\bS})_{\Sp/}$.
The counit of the adjunction above provides us with a functor
\[
\epsilon_{\TwSp(A,\tau)} : \Mod_{\End(a_! \bS )} \longto \TwSp(A,\tau)
\]
which is natural in $(\Cat_{\bS})_{\Sp /}$.
In~\cite{CCRY23}*{Theorem 7.13}, it is essentially shown that this counit map is an equivalence.
We want to extend this result slightly to be a natural equivalence, so that we know that the functors $f^*$, $f_!$, and $f_*$ on categories of twisted spectra correspond to the standard restriction, extension, and coextension of scalars on the level of module categories.
In particular, we want to exhibit a \emph{natural} identification of $\End(a_! \bS)$ with $\Th(\Omega \tau)$, so that this holds true.

\begin{rmk}\label{rmk:commutative_diagram}
In particular, given a pointed map $f : (A,a) \to (B,b)$ the diagram
\[
\begin{tikzcd}
  & \arrow[dl,"a_!"'] \Sp \arrow[dr,"b_!"] & \\
  \TwSp(A,f^* \sigma) \arrow[rr,"f_!"] & & \TwSp(B,\sigma)
\end{tikzcd}
\]
provides us with a morphism in $(\Cat_R)_{\Mod_R / }$. It then follows that the diagram
\[
\begin{tikzcd}
\Mod_{\End(a_! \bS)} \arrow[r] \arrow[d,"\End(f_! \bS)_!"] & \TwSp(A,f^* \sigma) \arrow[d,"f_!"] \\
\Mod_{\End(b_! \bS)} \arrow[r] & \TwSp(B,\sigma) \,.
\end{tikzcd}
\]
commutes.
Note that the left hand vertical map is precisely extension of scalars along the $\bE_1$-ring map $\End(f_! \bS) : \End(a_! \bS) \to \End(b_! \bS)$.
\end{rmk}

Let us state and prove this desired natural equivalence.
Note that $\Omega A$ is an $\bE_1$-group and that we can think of the map $\Omega \tau : \Omega A \to \Pic(R)$ as an object of $\Grp_{/\Pic(R)}$.
In particular, the inclusion $a \to A$ is corresponds to the inclusion of the unit in $\Omega A$.
For some convenient notation, let us denote the inclusion of the unit of an arbitrary group $G$ by $\iota_G$.
Another way of writing the endomorphism ring $\End(a_! \bS)$ is then $\End((B\iota_{\Omega A})_! \bS)$.

\begin{definition}
  We let $\bS[G]$ denotes the free presentable $\infty$-category of the grouplike $\bE_n$-space $G$.
  We recall that there is an adjunction
  \[
  \cS[-] : \Grp_{\bE_n} \leftrightarrows \Alg_{\bE_n}(\mathrm{Pr}^{L}) : \Pic \,,
  \]
  which might be found for example in~\cite{andogepner}*{Theorem 7.7}.
  Recall that the underlying $\infty$-category of $\cS[G]$ is $\Fun(G,\cS)$, but with symmetric monoidal structure given by the Day convolution product, rather than the pointwise one.
\end{definition}

\begin{lemma} \label{lem:BPicPrLSt_factorization}
  The inclusion $B\Pic(\bS) \to \Pr^{L}_{\St}$ factors as
  \[
  \begin{tikzcd}
    B\Pic(\bS) \arrow[rr] \arrow[dr] &  & \Pr^{L}_{\St} \\
    & \Mod_{\cS[\Pic(\bS)]}(\Pr^L) \arrow[ur,"{(\epsilon_{\Sp})_!}"'] &
  \end{tikzcd}
  \]
  \begin{proof}
    First, note that there is a natural factorization $\Pic \to \cS[\Pic(\bS)] \to \Sp$ of the inclusion $\Pic(\bS) \to \Sp$, which follows by the universal property of $\cS[\Pic(\bS)]$ as the initial presentably symmetric monoidal $\infty$-category receiving an $\bE_{\infty}$-map from $\Pic(\bS)$.
    It then follows that there is a map $B \Pic(\bS) \to \Mod_{\cS[\Pic(\bS)]}(\on{Pr}^{L}_{\St})$, since there is a group map $\Pic(\bS) \to \Aut(\cS [\Pic(\bS)])$ by functoriality of $\Pic(\bS) \mapsto \cS[\Pic(\bS)]$.
  \end{proof}
\end{lemma}

\begin{proposition} \label{prop naturality of shit}
  The functors
  \[
  \Th : \Grp_{/\Pic(\bS)} \longto \Alg \och \End((B\iota_{(-)})_! \bS) : \Grp_{/\Pic(\bS)} \longto \Alg \,.
  \]
  are naturally equivalent.
  \begin{proof}
    We will divide up the proof in three steps.
    \begin{description}
      \item[Step 1] We first show that the two functors both factor through functors of the type
      \[
      \on{Grp}_{/\Pic(\bS)} \to \Alg(\cS[\Pic(\bS)]) \,.
      \]
       The counit of the $(\cS[-] \vdash \Pic)$-adjunction provides us with a natural symmetric monoidal functor
       \[
       \epsilon_{\Sp} : \cS[\Pic(\bS)] \to \Sp \,.
       \]
      For the Thom spectrum functor, note that there is a factorization
      \[
      \Grp_{/ \Pic(\bS)} \xrightarrow {\Th_{\Pic}} \Alg(\cS[\Pic(\bS)]) \overset{\Alg(\epsilon_{\Sp})}\longto \Alg \,,
      \]
       since $\cS[\Pic(\bS)]$ is the universal symmetric monoidal category admitting an $\bE_{\infty}$-map from $\Pic(\bS)$.
       The functor $\Th_{\Pic}$ is then just the functor that sends $\alpha \in \Grp_{/\Pic(\bS)}$ to the colimit
       \[
       \Th_{\Pic}(\alpha) = \colim_G\left(G \overset{\alpha}\to \Pic(\bS) \to \cS[\Pic(\bS)] \right) \,.
       \]
       For the second functor in the statement of the theorem, it is clear, essentially by definition, that we have a factorization
       \[
       \Grp_{/\Pic(\bS)} \longto (\on{Pr}^{L}_{\St})_{\Sp / } \overset{\End(1_{(-)})}\longto \Alg \,,
       \]
       where the first functor sends $\alpha : G \to \Pic(\bS)$ to the left adjoint functor $(B\iota_G)_! : \Sp \to \TwSp(BG,B\alpha)$.
       Moreover, by Lemma~\ref{lem:BPicPrLSt_factorization} the first functor in the composition above factors as
      \[
      \Grp_{/ \Pic(\bS)} \simeq {\cS_*^{\geq 1}}_{/ B \Pic(\bS)} \overset{F_{\Pic}}\longto  (\Mod_{\cS[\Pic(\bS)]}(\on{Pr}^{L}))_{\cS[\Pic(\bS)]/} \overset{(\epsilon_{\Sp})_!}\longto  (\on{Pr}^{L}_{\St})_{\Sp / } \,.
      \]
      Here, the functor $F_{\Pic}$ is the functor
      \[
      F_{\Pic}(\alpha) = \colim_{BG} \left(BG \longto B\Pic(\bS) \longto \Mod_{\cS[\Pic(\bS)]}(\mathrm{Pr}^{L})\right) \,.
      \]
      Next, we claim that after composing with the functor $\End(1_{(-)})$, the above argument implies a factorization of the assignment  $\Grp_{/ \Pic(\bS)} \to \Alg(\cS[\Pic(\bS)]) \to \Alg(\Sp)$. This can be seen from the commutativity of the diagram
      \begin{equation} \label{Benchen}
      \begin{tikzcd}
        \Grp_{/\Pic(\bS)} \arrow[r] \arrow[dr] & (\mathrm{Pr}^{L}_{\St})_{\Sp /} \arrow[r,"\End(1_{(-)})"] & \Alg \\
        & \Mod_{\cS[\Pic(\bS)]}(\mathrm{Pr}^L)_{\cS[\Pic(\bS)]/} \arrow[r,"\End^{\cS[\Pic(\bS)]}(1_{(-)})"] \arrow[u,"(\epsilon_{\Sp})_!"'] & \Alg(\cS[\Pic(\bS)]) \arrow[u,"\Alg(\epsilon_{\Sp})"'] \,.
      \end{tikzcd}
      \end{equation}
      where the right hand square commutes due to the naturality properties of the endomorphism functor established in~\cite{HA}*{Section 4.8.3-4.8.5}.
      \item[Step 2] So far, we concluded that there are morphisms
      \[
      \Th_G : \Grp_{/G} \longto \Alg(\cS[G]) \och \End_G : \Grp_{/G} \longto \Alg(\cS[G])
      \]
      for $G = \Pic(\bS)$ factoring the functors in the statement.
      In fact, such functors exist for any group $G$.
      We now show that these are natural in the group $G$.
      That is, we would like to show that if $f: G \to H$ is a map of grouplike $\bE_{\infty}$-spaces, then
      there will be an induced commutative diagrams
        \[
          \begin{tikzcd}
            \on{Grp}_{/G} \arrow[r,"\Th_{G}"] \arrow[d] & \Alg(\cS[G]) \arrow[d, "f_!"] \\
            \on{Grp}_{/H} \arrow[r,"\Th_{G'}"] & \Alg(\cS[H]),
          \end{tikzcd}
          \och
          \begin{tikzcd}
            \on{Grp}_{/G} \arrow[r,"\End_{G}"] \arrow[d] & \Alg(\cS[G]) \arrow[d, "f_!"] \\
            \on{Grp}_{/H} \arrow[r,"\End_{G'}"] & \Alg(\cS[H]) \,.
          \end{tikzcd}
          \]
      For the first diagram, note that     the horizontal arrows are induced by the straightening equivalences $\cS_{/G} \simeq \cS[G]$ and $\cS_{/H} \simeq \cS[H]$, which are monoidal by ~\cite[Corollary D]{ramzi} and are natural by definition. The commutativity of the second diagram reduces again to commutativity of the diagram (~\ref{Benchen}).
      \item[Step 3] Once, this is understood, we argue as in Proposition 7.8 of ~\cite{CCRY23}. We will show that the identity map $G \to G$ is sent to the terminal object of  $\Alg(\cS[G])$. We already know from ~\cite[Theorem 7.13]{CCRY23}
     that these two functors, $\Th_{G}$ and $\End((B\iota_{(-)})_! \bS)$ agree pointwise.
     Thus, it is enough to show that the  functor $\Th_{G}: \Grp_{/G} \to \Alg(\cS[G])$ sends the identity map $G \to G$ to the terminal object.
     Note that this is true because we can identify $\Th_{G}$ with the straightening equivalence, so terminal objects get sent to terminal objects, under the equivalence.
     We now defer to the lemma below - we set $\cC = \Grp(\cS)$ and $f: \Grp(\cS) \to \Cat_{\infty}$ to be the functor $G \mapsto \Alg(\cS[G])$.
     Since we now know that the two functors $\Th_{G}(-)$ and $\End(B(-)_!(*))$ satisfy the criteria, they are identified.
    \end{description}
\end{proof}

\begin{lemma}
    Let $\cC$ be any category and $f$ any functor $\cC \to \Cat_{\infty}$.
    Then the category of natural transformations
   \[
   \cC_{/(-)} \to f
   \]
   sending $id_x$ to terminal objects of $f(x)$ is contractible. Moreover, any such natural transformation is terminal in this category.
    \end{lemma}

    \begin{proof}
        Natural transformations $\cC_{/(-)} \to f$ can be identified with (non-Cartesian) sections of the cocartesian fibration $F$ classifying $f$.
        Informally, the section $\sigma$ sends some element $y$ to a pair $(y,\beta(y) \in f(y))$, from which the natural transformation $\tau : \cC_{/(-)} \to f$ corresponding to $\sigma$ is such that
        \[
        \tau_x : \cC_{/x} \to f(x) \,, \quad (g : y \to x) \mapsto g_*(\beta(y)) \,,
        \]
        where $g_* : f(y) \to f(x)$ is the functor induced by evaluating $f$ on $g$.
        Since we are assuming that we are dealing with natural transformations that send the identity to terminal objects, we must then have that $\beta(y)$ is terminal in $f(y)$ for every $y$.
        Since terminal objects in the $\infty$-category of sections $\Gamma(\cC, F)$ may be understood pointwise, it follows that the section $\sigma$ itself is terminal in this $\infty$-category.
        Hence it is determined uniquely up to contractible choice.
    \end{proof}
\end{proposition}

\begin{theorem} \label{thm:TwSp_Mod_natural}
  Let $A$ be a connected space, and consider a haunt $\tau : A \to B\Pic(\bS)$.
  There is a natural equivalence
  \[
  \TwSp(A,\tau) \simeq \Mod_{\Th(\Omega \tau)} \,.
  \]
\end{theorem}

\begin{rmk}
  The above result gives a conceptual reason for why the 2-periodic sphere spectrum, and modules over this ring, seem to be central to work in symplectic Floer theory, see for example~\cite{AB21}.
  Indeed, the 2-periodic sphere spectrum is precisely the Thom spectrum of the loop of the map
  \[
  S^1 \simeq U(1) \longto U \simeq B(\bZ \times BU) \longto B\Pic(\bS)
  \]
  where we are using the inclusion of $U(1)$ into $U$, Bott periodicity, and the complex J-homomorphism~\cite{LurRotKthy}.
  Interestingly, the loop of this map is actually $\bE_2$.
  Hence, the $\infty$-category of modules over the 2-periodic sphere spectrum is monoidal in its own right.
  This raises questions about how this monoidal structure interacts with the monoidal structure
  \[
  \otimes : \TwSp(S^1,\tau) \otimes \TwSp(S^1, \sigma) \longto \TwSp(S^1 , \tau + \sigma) \,,
  \]
  on twisted spectra over the circle that we introduced in Section~\ref{sec:tensor_prod_TwSp}.
\end{rmk}

As corollaries, we can conclude that the pullback functors, and their left and right adjoints, on categories of twisted spectra correspond to the standard restriction, extension, and coextension of scalars on module categories.

\begin{corollary}
  Assume that the spaces $A$ and $B$ are pointed and connected and let $\sigma : B \to B\Pic(\bS)$ be a twist.
  Let $f : A \to B$ be a map of based spaces.
  Then the diagram
    \[
    \begin{tikzcd}
      \Mod_{\Th(\Omega(f^* \sigma))} \arrow[r,"\Th(\Omega f)_!"] \arrow[d,"\simeq"] & \Mod_{\Th(\Omega \sigma)} \arrow[d,"\simeq"] \\
      \TwSp(A, f^* \sigma) \arrow[r,"f_!"] & \TwSp(B,\sigma)
    \end{tikzcd}
    \]
    commutes.
    \begin{proof}
      Let $f : A \to B$ be a based map of connected spaces. Then the diagram
      \[
      \begin{tikzcd}
        \Omega A \arrow[rr,"\Omega f"] \arrow[dr,"\Omega(f^* \sigma)"'] & & \Omega B \arrow[dl,"\Omega \sigma"] \\
        & \Pic(\bS) &
      \end{tikzcd}
      \]
      provides us with a morphism in $\Grp_{/\Pic(\bS)}$. Because of the naturality of the functors, the diagram
      \[
      \begin{tikzcd}
        \Th(\Omega(f^* \sigma)) \arrow[d] \arrow[r,"\Th(\Omega f)"] & \Th(\Omega \sigma ) \arrow[d] \\
        \End(a_! \bS) \arrow[r,"\End(f_! \bS)"] & \End(b_! \bS)
      \end{tikzcd}
      \]
      commutes in $\Alg$. The result then follows by using Remark~\ref{rmk:commutative_diagram}.
    \end{proof}
\end{corollary}

\begin{corollary}
  Assume that the spaces $A$ and $B$ are pointed and connected and let $\sigma : B \to B\Pic(\bS)$ be a twist.
  Let $f : A \to B$ be a map of based spaces. Then:
  \begin{enumerate}
    \item The diagram
      \[
      \begin{tikzcd}
        \Mod_{\Th(\Omega\sigma)} \arrow[r,"\Th(\Omega f)^*"] \arrow[d, "\simeq"] & \Mod_{\Th(\Omega(f^* \sigma))} \arrow[d,"\simeq"] \\
        \TwSp(B,\sigma) \arrow[r,"f^*"] & \TwSp(A,f^* \sigma)
      \end{tikzcd}
      \]
      commutes.
    \item The diagram
      \[
      \begin{tikzcd}
        \Mod_{\Th(\Omega(f^* \sigma))} \arrow[r,"\Th(\Omega f)_*"] \arrow[d,"\simeq"] & \Mod_{\Th(\Omega \sigma)} \arrow[d,"\simeq"] \\
        \TwSp(A, f^* \sigma) \arrow[r,"f_*"] & \TwSp(B,\sigma)
      \end{tikzcd}
      \]
      commutes.
  \end{enumerate}
\end{corollary}

\section{Homology and Cohomology Invariants of Twisted Spectra}

It is natural to ask what sorts of invariants can be extracted from a twisted spectrum. As these should often arise from Floer homotopical data, it is would be useful to extract honest abelian groups which can say something about the context at hand. In this section, we investigate what sort of invariants one can attach to a twisted spectrum. As we will see, it will not be so straightforward to define global invariants as it is in the untwisted case.

\subsection{Non-existence of global homotopy invariants}

Let $B$ be a fixed space.
To a parametrized spectrum $X : B \to \Sp$, one may associated its \emph{homology type} and \emph{cohomology type}
\[
p_!(X) \simeq \colim_{B}X \in \Sp \och p_*(X) \simeq \lim_{B}X \in \Sp \,,
\]
respectively.
Now let $\tau: B \to B \Pic$ be the data of a haunt over $B$, and let $X \in \TwSp(B, \tau)$ be a twisted parametrized spectrum.
One is led to ask the following natural question.

\begin{question}
Do globally defined homology and cohomology invariants exist in general for $X$ as they do in the untwisted case?
\end{question}

The answer, as we shall see, is a resounding no.
Namely, there is generally \emph{no canonical way} to associate homology and cohomology invariants to a twisted spectrum, in stark contrast to ordinary parametrized spectra.
In this section, we propose a simple conceptual explanation for this phenomenon: that the $\infty$-category of twisted parametrized spectra is fibered over the $\infty$-category $\cS_{/ B \Pic(\bS)}$ of twisted spaces, whose structure adds constraints to the resulting functorialities that may arise in the total space.

\begin{porism} \label{henven}
 As a consequence of the discussion in Section \ref{prop:functoriality_twisted_spectra}, the assignment $(B, \tau) \mapsto \TwSp(B, \tau)$ may be packaged into a symmetric monoidal functor
\[
\TwSp^*:  \cS_{/ B \Pic(\bS)}^{\oop} \to \on{Pr}^{L, R}_{\St}\,,
\]
where the target denotes the $\infty$-category of stable presentable $\infty$-categories with morphisms which are both left and right adjoints.
\end{porism}

In particular, let $\tau: A \to B \Pic(\bS)$ be a twist, and let us try to define a functor $\TwSp(A, \tau) \to \Sp$. If we try to define a functor $p_!: \TwSp(A, \tau) \to \Sp$ or $p_*: \TwSp(A, \tau) \to \Sp$ via the functoriality of Porism  \ref{henven}, we immediately see that we need the collapse map $p: A \to *$ to sit in the diagram
\[
\begin{tikzcd}
  A \arrow[rr,"p"] \arrow[dr,"\tau"'] && * \arrow[dl,"0"] \\
  & B \Pic(\bS) \,. &
\end{tikzcd}
\]
In particular, the equivalence $0 \simeq p^*(0) \simeq \tau$ is forced upon us.
This many not appear so interesting at first glance.
Something slightly interesting is hidden in the notation though.
Indeed, while the above only make sense when the twist is nullhomotopic, the invariants that we get depend on how we trivialize the twist.
Note that the trivialization corresponds to a particular equivalence $\TwSp(A,\tau) \simeq \Sp^{A}$ with the category of spectra parametrized over $A$.

\subsection{Homology and Cohomology Invariants}

Now, it could of course happen that we are dealing with a twisted spectrum with a non-trivializable twisting $\tau$, in which case our twisted spectrum does not have a globally defined homotopy invariant.
We are not completely out of luck, though.
Indeed, we might still have (co)homology invariants with respect to other generalized cohomology theories.
Let $R$ be an $\bE_\infty$-ring, so that we have an induced map $B\Pic(\bS) \to B\Pic(R)$ of Picard spaces.
Now, if the composition $\tau_R : A \to B\Pic(\bS) \to B\Pic(R)$ is nullhomotopic, then it makes sense to talk about $R$-homology and  $R$-cohomology of $X$, completely analogous to before.
Again the (co)homology will depend on our choice of trivialization of $\tau_R$, though.

\begin{proposition}
  The space of trivializations of $\tau_R : A \to B\Pic(R)$ is a torsor for the group $\Map(A, \Pic(R))$.
\end{proposition}

\begin{proof}
    Let $\tau \in  \pi_0( \Map(A, B \Pic))$ be an equivalence class of haunts, such that there is a null homotopy $\iota: \tau \simeq 0$.
    Note first that every $\alpha: A \to \Pic$ classifies an automorphism of the trivial haunt.
    In particular
    \[\Aut_{\Map(A, B \Pic)}(0) \simeq \Map(A, \Pic(R)) \,.
    \]
    Thus, we may post-compose $\iota$ with any such automorphism to obtain another trivialization $\iota \circ \alpha$.
    We see therefore  that there is a transitive action of $\Map(A, \Pic(R))$ on the space of trivializations.
    Now, since the group $\Map(A, \Pic)$, being the automorphism group of the trivial twist, acts  transitively on itself, we conclude that the action on the space of trivializations is itself transitive.
\end{proof}
\noindent Each one of these trivializations $\alpha$ corresponds to an equivalence
\[
\mathrm{TwMod}_R(A,\tau_R) \overset{\alpha}\simeq \Fun(A,\Mod_R)
\]
that we by abuse of notation will also denote $\alpha$.

\begin{definition}
  Let $X \in \TwSp(A,\tau)$ be a twisted spectra for a haunt satisfying that $\tau_R : A \to B\Pic(R)$ is nullhomotopic.
  Given a choice of trivialization $\alpha$ of $\tau_R$, we write
  \[
  R^{\alpha}_i = \pi_i\left(\colim_A \alpha(X) \right) \och R^{i}_{\alpha} = \pi_i \left(\lim_A \alpha(X) \right)
  \]
  and refer to this as the $(R,\alpha)$-homology and $(R,\alpha)$-cohomology of $X$, respectively.
\end{definition}

Let us look at some examples, where the cohomology theories that we keep in mind are those that appear in Table~\ref{table:homotopy}.

\begin{example}[Invariants of Twisted Spectra over the Circle]
  Twisted spectra over the circle have $\KU$-invariants whenever we are dealing with a haunt corresponding to an $n \in \bZ \cong \pi_1(B\Pic(\bS))$ which is zero mod $2$.
  This follows from the maps
  \[
  \pi_1 B\Pic(\bS) \longto \pi_1 B\Pic(\KU)
  \]
  being reduction mod $2$.
  In the cases that the map $\tau_{\KU} : S^1 \to B\Pic(\KU)$ is trivializable, note that there are essentially two different trivializations, since $\pi_1 \Pic(\KU) \cong \bZ/ 2$.
\end{example}

\begin{example}[Invariants of Twisted Spectra over the 3-Sphere]
  Consider the non-trivial haunt $\tau : S^3 \to B\Pic(\bS)$.
  Since $\pi_3 B\Pic(H\bZ)$, the map $\tau_{H\bZ} : S^3 \to B\Pic(\bS)$ is trivializable, and since $\pi_0 \Map(S^3,\Pic(H\bZ)) \cong \pi_3 \Pic(H\bZ)$, there is an essentially unique way to trivialize the twist. The $H\bZ$-invariants for a twisted spectrum over $S^3$ is hence uniquely determined.
\end{example}

\begin{example}[Invariants of Twisted Spectra from Seiberg--Witten Floer Theory]
  A haunt $\tau : P \to B\Pic(\bS)$ obtained from a complex spin manifold $(Y,\mathfrak{s})$ always admits $\KU$-invariants.
  Note that the induced map
  \[
  \pi_i \tau_{\KU} : \pi_i P \longto \pi_i U \longto \pi_i B\Pic(\bS) \longto \pi_i B\Pic(\KU) \,,
  \]
  is zero for all $i$.
  Indeed, it is clear that the map is zero for even $i \geq 2$ since $\pi_i U \cong 0$ for all even $i$, and that it is zero for odd $i\geq 3$ since $\pi_i B\Pic(\KU) \cong 0$ for all odd $i$ in this range.
  We also claim that the map induced for $i=1$ is zero.
  The rough argument uses that isomorphism classes of complex spin structures on $Y$ is in one-to-one correspondence with complex line bundles over $Y$~\cite{KM07}*{Proposition 1.1.1}.
  Indeed, given one complex spin structure on $Y$ with spinor bundle $S$, all other complex spin structures can be obtained by tensoring $S \otimes L$ with a hermitian line bundle over $Y$.
  The first Chern character of $S \otimes L$ is known to be
  \[
  c_1(S \otimes L) = c_1(S) + 2 c_1(L) \,.
  \]
  Since the tangent bundle of an oriented 3-manifold is always trivial, we can always find a complex spin structure on $Y$, namely the one for which the spinor bundle is trivial.
  Hence, the first Chern character of any complex spin structure is necessarily even, and this will imply that the corresponding haunt is such that the induced map
  \[
  \pi_1 \tau : \pi_1 P \longto \pi_1 B\Pic(\bS) \cong \bZ
  \]
  takes values in $2\bZ$. Since the map $\pi_1 B\Pic(\bS) \to \pi_1 B\Pic(\KU)$ is mod 2, we conclude that $\pi_1 \tau_{\KU}$ is zero.
\end{example}

Moreover, we want to point out the following result of Maegawa which tells us that twisted spectra coming from haunts that factor through the complex J-homomorphism, which are common in Floer homotopy theory, always admit $R$-(co)homology group with respect to any even $\bE_2$-ring $R$.

\begin{proposition}[\cite{maegawa24}*{Theorem 3.1}]
  Let $B$ be a CW-complex and let $\tau : B \to B(\bZ \times \BU) \to B\Pic(\bS)$ be a haunt that factor through the complex J-homomorphism.
  Then $\tau_{R} : B \to B\Pic(R)$ is nullhomotopic for any even periodic $\bE_2$-ring spectrum.
  Moreover, the nullhomotopy is determined by essentially uniquely by the choice of an $\bE_1$-complex orientation on $R$.
\end{proposition}

\section{The Total Category of Twisted Spectra} \label{sec:total_category}

The discussion on tensor products of twisted spectra suggests that there is a symmetric monoidal structure on some category of twisted spectra on a fixed space, where we let the twists vary.
In this section, we introduce this total category of twisted spectra over a space, show that is has the wanted closed symmetric monoidal structure, and discuss pullbacks and their adjoints on these sorts of categories.
For all intents and purposes, we will essentially show that the total categories admit a 6-functor formalism.

\subsection{The Total Category of Twisted Spectra}

We start by simply introducing the wanted total category of twisted spectra over a fixed space.

\begin{definition}
  The \emph{total category of twisted spectra} over $A$ is defined as
  \[
  \TwSp(A) = \colim\left(\TwSp(A,-) : \Map(A,B\Pic(\bS)) \longto \mathrm{Pr}^{L}_{\St}  \right) \,.
  \]
\end{definition}

Note that by ambidexterity, the total category is also obtained as the limit of the diagram.

\begin{rmk}
  There is another potential definition of a total category for twisted spectra that we can form by looking at the total space of the fibration classified by the global sections functor
  \[
  \colim : \Map(A,B\Pic(\bS)) \longto \Cat_\infty , \quad \tau \mapsto \TwSp(A,\tau) \,.
  \]
  Heuristically, we think of objects in this category as a pair $(\tau,X)$ where $X \in \TwSp(A,\tau)$.
  We will not study this category in this paper, mostly because it does not have the wanted properties that are looking for.
  In particular, note that the lax version of the total category is not stable.
\end{rmk}

Let us look at a possible example of an object in one of these total categories that appear in nature.

\begin{example}[Seiberg--Witten Floer Theory]
  Recall Example~\ref{ex:SWF} and~\ref{ex:SWF_twisted_spectrum}, where we sketched the haunt and twisted spectrum associated to an oriented riemannian 3-manifold equipped with a complex spin structure $(Y,\mathfrak{s})$.
  It is well-understood how the haunt $\tau_{(Y,\mathfrak{s})} : P \to B\Pic(\bS)$ changes when we vary the complex spin structure $\mathfrak{s}$.
  We expect that, given an oriented riemannian 3-manifold $Y$, it should be possible to construct an object $\mathrm{SWF}(Y)$ in the total category $\TwSp(P)$ of twisted spectra over the Picard torus $P = H^1(Y;\bR) /H^1(Y;\bZ)$ in such a way that
  \[
  \SWF(Y)_{\tau} = \begin{cases} \SWF(Y,\mathfrak{s}) & \text{if there is a (necessarily unique) $\mathfrak{s}$ such that $\tau_{(Y,\mathfrak{s})} = \tau$} \\ 0 & \text{otherwise.} \end{cases}
  \]
  That such a ``total twisted spectrum'' might be interesting to study is hinted at by the fact that one in monopole Floer homology tend to consider all possible isomorphism classes of complex spin structures on a 3-manifold~\cite{KM07}*{Chapter 3}.
\end{example}

\subsection{The Total Pullback Functors and Their Adjoints} \label{sec:total_pullback}

In this section, we show that the pullback functors we introduced in Section~\ref{sec:pullback} assemble into a pullback functor on the level of total categories, and that this pullback functor admits a right adjoint, but not necessarily a left adjoint.
Now, there are a few maps that we could denote $f^*$ for a map $f : A \to B$ of spaces, so to avoid overlap we will start this section by introducing the following notation.

\begin{notation}
  Given a map $f : A \to B$ of spaces, the pre-composition map induced by applying the functor $\Map(-,B\Pic(\bS))$ will be denoted
  \[
  B\Pic(f) : \Map(B,B\Pic(\bS)) \longto \Map(A,B\Pic(\bS)) \,.
  \]
\end{notation}

We also introduce the following definition of an auxiliary category that will be of interest.

\begin{definition}
  Let $f : A \to B$ be a map of spaces. We write
  \[
  \TwSp(A,f^*) = \colim\left(\Map(B,B\Pic(\bS)) \overset{B\Pic(f)}\longto \Map(A,\Br(\bS)) \overset{\TwSp(A,-)}\longto \mathrm{Pr}^{L}_{\St} \right)\,.
  \]
  This is essentially the subcategory of the total category $\TwSp(A)$ generated by twists that are in the image of $B\Pic(f)$.
  Note that the inclusion $\TwSp(A,f^*) \to \TwSp(A)$ can be identified with the functor $(B\Pic(f))_!$ of the second part of Proposition~\ref{prop:left_right_adjoints_ambidexterity}.
\end{definition}

Note that we can think of all the pullback functors $f^* : \TwSp(B,\tau) \to \TwSp(A,f^* \tau)$ for various twists $\tau$ as assembling into a natural transformation
\[
\mathrm{Nat}(f^*) : \TwSp(B,-) \longto \TwSp(A,-) \circ B\Pic(f)
\]
of functors $\Map(B,B\Pic(\bS)) \to \Pr^{L}_{\St}$. Let us consider the following construction.

\begin{construction}
Let $f : A \to B$ be a map of spaces. We can construct an induced \emph{total pullback functor} as a composition
\[
f^* : \TwSp(B) \longto  \TwSp(A,f^*) \longto \TwSp(A) \,
\]
as follows:
\begin{itemize}
  \item The first functor $\TwSp(B) \to \TwSp(A,f^*)$ is the one obtained by feeding the natural transformation $\mathrm{Nat}(f^*)$ through the colimit functor
  \[
  \colim : \Fun(\Map(B,B\Pic(\bS)),\mathrm{Pr}^{L}_{\St}) \longto \mathrm{Pr}^{L}_{\St} \,.
  \]
  \item The second functor $\TwSp(A,f^*) \to \TwSp(A)$ is the inclusion of the full subcategory.
\end{itemize}
\end{construction}

We now arrive at the question to the existence of adjoints to this total pullback functors. Similarly to before, note that the different functors $f_!$ on the categories of twisted spectra for various twists assemble into a natural transformation
\[
\mathrm{Nat}(f_!) : \TwSp(A,-) \circ B\Pic(f) \longto \TwSp(B,-)
\]
of functors $\Map(B,B\Pic(\bS)) \to \Pr^{L}_{\St}$. Of course, the different functors $f_*$ on categories of twisted spectra for various twists also assemble into a natural transformation
\[
\mathrm{Nat}(f_*) : \TwSp(A,-) \circ B\Pic(f) \longto \TwSp(B,-)
\]
of functors $\Map(B,B\Pic(\bS),\Pr^{R}_{\St}) \to \Pr^{R}_{\St}$.
Note that we need use $\Pr^{R}_{\St}$ here since the functors $f_*$ are right adjoints.

\begin{proposition}
  The following holds:
  \begin{enumerate}
    \item The functor we get from feeding $\mathrm{Nat}(f^*)$ into the limit functor
    \[
    \lim : \Fun(\Map(B,B\Pic(\bS)),\mathrm{Pr}^{R}_{\St}) \longto \mathrm{Pr}^{R}_{\St}
    \]
    is equivalent to the functor $\TwSp(B) \to \TwSp(A,f^*)$ appearing above.
    \item The functor we get from feeding $\mathrm{Nat}(f_!)$ into the colimit functor
    \[
    \colim : \Fun(\Map(B,B\Pic(\bS)),\mathrm{Pr}^{L}_{\St}) \longto \mathrm{Pr}^{L}_{\St}
    \]
    is left adjoint to $\TwSp(B) \to \TwSp(A,f^*)$
    \item The functor we get from feeding $\on{Nat}(f_*)$ into the limit functor
    \[
    \lim : \Fun(\Map(B,B\Pic(\bS)),\mathrm{Pr}^{R}_{\St}) \longto \mathrm{Pr}^{R}_{\St}
    \]
    is right adjoint to $\TwSp(B) \to \TwSp(A,f^*)$.
  \end{enumerate}
  \begin{proof}
   The first statement is due to ambidexterity in $\Pr^{L}_{\St}$ and $\Pr^{R}_{\St}$. The second and third statement then following from Propositions 1.5.6 and 1.5.7.
  \end{proof}
\end{proposition}

Because of the above result, the existence of a left and right adjoint to the total pullback functor essentially comes down to the existence of adjoints to the inclusion $\TwSp(A,f^*) \to \TwSp(A)$.
It turns out that the wanted right adjoint to the pullback on total categories always exists, but that the left adjoint only exist under some very strict finiteness assumptions.

\begin{proposition} \label{prop:total_pullback_has_right_adjoint}
 The total pullback functor $f^* : \TwSp(B) \to \TwSp(A)$ has a right adjoint.
 \begin{proof}
Recall that the inclusion $\TwSp(A,f^*) \to \TwSp(A)$ is simply the functor $(B\Pic(f))_!$ of the second part of Proposition~\ref{prop:left_right_adjoints_ambidexterity}.
Hence, it necessarily has a right adjoint.
 \end{proof}
\end{proposition}

Note that the same argument cannot be run for the left adjoint; the functor $\TwSp(A,f^*) \to \TwSp(A)$ does not in general admit a left adjoint, unless we have some pretty strict finiteness conditions on the fiber of $f$.

\begin{proposition} \label{prop:total_pullback_no_left_adjoint}
 Let $f : A \to B$ be a map of spaces such that the induced map $B\Pic(f)$ is an equivalence, then we have an induced left adjoint functor
 \[
 f_! : \TwSp(A) \to \TwSp(B) \,.
 \]
 Moreover, the right adjoint to this functor is $f^*$.
 \begin{proof}
   Note that if $B\Pic(f)$ is an equivalence, then the functor $B\Pic(f)_!$ is an equivalence, so it certainly has a left adjoint, namely its inverse.
 \end{proof}
\end{proposition}

We also prove Beck--Chevalley in this setting.

\begin{proposition}[Beck--Chevalley] \label{prop:BC_total_cat}
  Suppose that
  \[
  \begin{tikzcd}
  A \arrow[r,"j"] \arrow[d,"i"] & B \arrow[d,"f"] \\
  C \arrow[r,"g"] & D
  \end{tikzcd}
  \]
  is a pullback square of spaces in which $i$ and $f$ are such that $B\Pic(i)$ and $B\Pic(f)$ are equivalences. Then, the canonical natural transformation
  \[
  i_! j^* \longto g^* f_!
  \]
  is an equivalence of functors $\TwSp(B) \to \TwSp(C)$.
  \begin{proof}
    It is enough to check that the functors $\TwSp(B,\sigma) \to \TwSp(C)$ are naturally equivalent for all choices of twist $\sigma$.
    Fixing a twist $\sigma$, the upper composition corresponds to the functor
    \[
    \TwSp(B,\sigma) \overset{j^*}\longto \TwSp(A,j^* \sigma) \simeq \TwSp(A,i^* \rho) \overset{i_!}\longto \TwSp(C,\rho)
    \]
    where the middle equivalence is coming from the fact that $B\Pic(i)$ is an equivalence, so that there is an essentially unique $\rho \in \Map(C,B\Pic(\bS))$ such that $i^* \rho \simeq j^* \sigma$.
    Similarly, the lower composition corresponds to the functor
    \[
    \TwSp(B,\sigma) \simeq \TwSp(B,f^* \pi) \overset{f_!}\longto \TwSp(D,\pi) \overset{g^*}\longto \TwSp(C,g^* \pi)
    \]
    where the equivalence at the start is comping from the fact that $B\Pic(f)$ is an equivalence, so that there is an essentially unique $\pi \in \Map(D,B\Pic(\bS))$ such that $f^*\pi \simeq \sigma$.
    Next, we note that $g^* \pi \simeq \rho$.
    Indeed, since $B\Pic(i)$ is an equivalence, this follows from
    \[
    i^* g^* \pi \simeq j^* f^* \pi \simeq j^* \sigma \simeq i^* \rho \,.
    \]
    The wanted result now follows from applying Proposition~\ref{prop:Beck_Chevalley_preliminary} to the pullback diagram
    \[
    \begin{tikzcd}
    (A,\tau) \arrow[r,"j"] \arrow[d,"i"] & (B,\sigma) \arrow[d,"f"] \\
    (C,\rho) \arrow[r,"g"] & (D,\pi)
    \end{tikzcd}
    \]
    in the category $\cS_{/B\Pic(\bS)}$. Here, $\tau$ is the necessarily unique twist on $A$ such that $\tau \simeq j^* f^* \pi \simeq i^* g^* \pi$.
  \end{proof}
\end{proposition}

\begin{rmk}
  We note that if $f : A \to B$ is a map of spaces such that the induced map $B\Pic(f)$ is an equivalence, then the map $f$ is itself very close to being an equivalence itself.
  Indeed, such a map must necessarily induce isomorphisms of homotopy groups in degrees $\geq 2$, since $\Omega \Pic(\bS) \simeq \Omega^\infty \bS$.
  In fact, the authors of this paper has not yet come up with an example of a map $f : A \to B$ for which $B\Pic(f)$ is an equivalence, but which is itself \emph{not} an equivalence.
\end{rmk}

\subsection{Symmetric Monoidal Structure on the Total Category}

We want to assemble the interior product to a symmetric monoidal structure on the total category of twisted spectra over the space $A$.
 Let us recall some important facts.
 For every space $A$, the mapping space $\Map(A,B\Pic(\bS))$ becomes an $\bE_\infty$-space by using the $\bE_\infty$-structure on $\Pic(\bS)$ and the cartesian symmetric monoidal structure on $\cS$, which makes every space into a co-$\bE_\infty$-spaces via its diagonal map.
 Note that the map $B\Pic(f)$ is always an $\bE_\infty$-map.

\begin{proposition} \label{prop:total_cat_symm_mon}
  The category $\TwSp(A)$ is endowed with a symmetric monoidal structure that we will denote as $\otimes_A$.
  This symmetric monoidal structure preserves small colimits in each variable separately.
  Moreover, this symmetric monoidal structure can be identified with the interior product described in Definition~\ref{defn:interior_tensor} in the sense that the diagram
  \[
  \begin{tikzcd}
    \TwSp(A,\tau) \otimes \TwSp(A,\sigma) \arrow[r,"\otimes"] \arrow[d] & \TwSp(A , \tau + \sigma) \arrow[d] \\
    \TwSp(A) \otimes \TwSp(A) \arrow[r,"\otimes_A"] & \TwSp(A)
  \end{tikzcd}
  \]
  commutes.
  \begin{proof}
    Note that $\Map(A,B\Pic(\bS))$ is an $\bE_\infty$-space by the interior product and that the functor
    \[
    \TwSp(A,-) : \Map(A,B\Pic(\bS)) \longto \Fun(A,\mathrm{Pr}^{L}_{\St}) \longto \mathrm{Pr}^{L}_{\St}
    \]
    is lax symmetric monoidal. In other words, it is an $\bE_\infty$-algebra in the category $\Fun(\Map(A,B\Pic(\bS)),\mathrm{Pr}^{L}_{\St})$.
    Since the functor
    \[
    \colim : \Fun(\Map(A,B\Pic(\bS),\mathrm{Pr}^{L}_{\St}) \longto \mathrm{Pr}^{L}_{\St}
    \]
    is symmetric monoidal, this tells us that the colimit $\TwSp(A)$ is an $\bE_\infty$-algebra in $\Pr^L_{\St}$, that is, a presentably symmetric monoidal $\infty$-category.
    All that is left now is identifying this with the interior product on twisted spectra already defined in Definition~\ref{defn:interior_tensor}. Recall that the upper map is defined as the composition
    \[
    \TwSp(A,\tau) \times \TwSp(A,\sigma) \longto \TwSp(A,\tau) \otimes \TwSp(A,\sigma)  \simeq \TwSp(A \times A , \tau \boxtimes \sigma) \overset{\Delta^*}\longto \TwSp(A, \tau + \sigma)
    \]
    and the identifications done in the proof of Proposition~\ref{prop:tensor_twisted_categories} and the fact that commuting colimits is natural.
  \end{proof}
\end{proposition}

For reference, we also include the following lemma.

\begin{lemma} \label{lem:inclusion_symm_monoidal}
  Let $f : A \to B$ be a map of spaces. The category $\TwSp(A,f^*)$ is symmetric monoidal and the inclusion $\TwSp(A,f^*) \to \TwSp(A)$ is symmetric monoidal.
  \begin{proof}
    The map $B\Pic(f)$ is a map of $\bE_\infty$-spaces, so the symmetric monoidality of $\TwSp(A,f^*)$ follows from the same line of argument as the proof of Proposition~\ref{prop:total_cat_symm_mon}. That the inclusion is symmetric monoidal is clear since $\TwSp(A,f^*)$ is a subcategory of $\TwSp(A)$ which is closed under the tensor product.
  \end{proof}
\end{lemma}

\begin{definition}
  If $X \in \TwSp(A,\tau)$ is dualizable in the symmetric monoidal $\infty$-category $\TwSp(A)$, we will say that $X$ is \emph{Spanier--Whitehead dualizable}. We will refer to its dual $D^{SW} X$, which is necessarily a $\TwSp(A,-\tau)$-twisted spectrum, as the \emph{Spanier--Whitehead dual} of $X$.
\end{definition}

We note that, just as Spanier--Whitehead duality for parametrized spectra is not the ``correct'' notion of dualizablity in that setting, see~\cite{MS06}*{Part IV}, there is another form of dualizablity of twisted spectra that are more appropriate.
We will look at this in Section~\ref{sec:dualizability_twisted_spectra}.

\begin{proposition} \label{prop:pullback_strong_monoidal}
  The functor $f^* : \TwSp(B) \to \TwSp(A)$ is symmetric monoidal with respect to the interior product $\otimes$. That is, if $X$ and $Y$ are twisted spectra over $B$, then
  \[
  f^*(X \otimes Y) \simeq f^* X \otimes f^* B \,.
  \]

  \end{proposition}

  \begin{proof}
    Recall that the functor $f^* : \TwSp(B) \to \TwSp(A)$ factors as
    \[
    \TwSp(B) \longto \TwSp(A,f^*) \longto \TwSp(A) \,,
    \]
    so it is enough to show that both functors in the composition are symmetric monoidal.
    That the first functor is symmetric monoidal follows from the fact that the natural transformation $\mathrm{Nat}(f^*)$ is a symmetric monoidal transformation, which follows from Corollary~\ref{cor:pullback_symm_mon_prelim}, and that the colimit functor is symmetric monoidal. That the second functor is symmetric monoidal follows from Lemma~\ref{lem:inclusion_symm_monoidal}.
    \end{proof}

    \subsection{Projection Isomorphism on Total Category}

    In this section, we prove the projection isomorphism in the total category of twisted spectra over a fixed space.
    This essentially follows directly from the Beck--Chevalley result that we proved in Section~\ref{sec:total_pullback}.

\begin{proposition}[Projection isomorphism] \label{prop:proj_iso_total}
  Let $f : A \to B$ be a map of spaces such that $B\Pic(f)$ is an equivalence and let $X \in \TwSp(A)$ and $Y \in \TwSp(B)$. Then we have an equivalence
  \[
  f_!(f^* Y \otimes X) \simeq Y \otimes f_! X
  \]
  of twisted spectra over $B$.
  \begin{proof}
    Formally, we have that
    \begin{align*}
      f_! (f^* Y \otimes X) &\simeq f_! (\Delta_A^*(f^* Y \boxtimes X)) \\ &\simeq f_!(\Delta_A^*((f \times 1)^*(Y \boxtimes X))) \\ &\simeq f_! (((f \times 1) \circ \Delta_A)^*)(Y \boxtimes X) \\ &\simeq f_!((f,1)^*(Y \boxtimes X))
    \end{align*}
    while
    \begin{align*}
    Y \otimes f_! X &\simeq \Delta_B^* (Y \boxtimes f_! X) \\ &\simeq \Delta_B^*((1 \times f)_!(Y \boxtimes X)) \,.
    \end{align*}
    To finish the proof we need to use Proposition~\ref{prop:BC_total_cat} for the diagram
    \[
    \begin{tikzcd}
      A \arrow[d,"f"'] \arrow[r,"{(f,1)}"] & B \times A \arrow[d,"1 \times f"] \\
      B \arrow[r,"\Delta_B"] & B \times B
    \end{tikzcd}
    \]
    so that we get an equivalence
    \[
    f_! \circ (f,1)^* \simeq \Delta_B^* \circ (1 \times f)_! \,.
    \]
    Note that since $f$ satisfies the necessary assumption for applying Proposition~\ref{prop:BC_total_cat}, so does $1 \times f$ by observing that $\cofib(f) \simeq \cofib(1 \times f)$ and using the fiber sequence
    \[
    \Map(\cofib(1 \times f),B\Pic(\bS)) \longto \Map(B \times B,B\Pic(\bS)) \overset{B\Pic(1 \times f)}\longto \Map(B \times A , B\Pic(\bS)) \,.
    \]
  \end{proof}
\end{proposition}

\subsection{Closed Structure on the Total Category}

In this section, we show that the symmetric monoidal structure on $\TwSp(A)$ is closed, so that there exist functorial internal hom twisted spectra, which are right adjoint to the tensor product.
Moreover, we identify this with the internal hom twisted spectra functors that we considered in Section~\ref{sec:closed_individualtwists}.
First we record the following lemma, which is immediate from the adjoint functor theorem.

\begin{proposition}
    Let $\cC^{\otimes}$ be a presentably symmetric monoidal presentable $\infty$-category.
    Then $\cC^{\otimes}$ is closed in the sense that for each object $Y \in \cC$, the functor $- \otimes Y: \cC \to \cC $ admits a right adjoint.
\end{proposition}

Turning to the case of twisted spectra at hand, we set $Y \in \TwSp(A)$.
Then the functor $- \otimes Y$ has right adjoint, which we denote by $\hom_{\TwSp(A)}(Y, -)$, or simply $\hom(Y,-)$ when the context is clear.
Being an adjoint, we have the natural equivalence
\begin{equation}
    \map(X \otimes Y, Z) \simeq \map(X , \hom(Y,Z)) \,.
\end{equation}
We would like to show that this constructed internal hom is compatible with internal homs that we constructed in Section~\ref{sec:closed_individualtwists}.
Morally this should just follow by uniqueness of adjoints and the fact that the symmetric monoidal structure on the total category $\TwSp(A)$ is determined by the various pairings
\[
\TwSp(A, \tau)\otimes \TwSp(A, \sigma) \to \TwSp(A, \tau + \sigma) \,,
\]
as per Proposition~\ref{prop:total_cat_symm_mon}.
However, let us be a bit more precise about this.

\begin{proposition} \label{prop:closed_total_cat}
The symmetric monoidal structure of $\TwSp(A)$ is closed. Moreover, this closed structure can be identified by the definition of the interior homs of Section~\ref{sec:closed_individualtwists}. Indeed, let $Y \in \TwSp(A, \sigma)$ be a twisted spectrum viewed as an object in $\TwSp(A)$ via the canonical inclusion $i_{\sigma}: \TwSp(A, \sigma) \to \TwSp(A)$.
Then there exists a commutative diagram
\[
      \begin{tikzcd}
     \TwSp(A)  \arrow[r,"{\hom(X,-)}"] \arrow[d] & \TwSp(A) \arrow[d] \\
    \TwSp(A, , \tau + \sigma ) \arrow[r,"{\hom(X,-)}"] & \TwSp(A, \tau) \,,
  \end{tikzcd}
  \]
  for all twists $\tau$.
\end{proposition}

\begin{proof}
Let $B \tau$ denote the connected component of $\tau$ in the space $\Map(A, B \Pic(\bS))$.
Note that the inclusion $\iota_{\tau} : B\tau \hookrightarrow \Map(A,B\Pic(\bS))$ induces a map $(\iota_{\tau})_{!}: \TwSp(A, \tau) \to \TwSp(A)$; we can identify this with the  ``inclusion map" mentioned in the statement.
Moreover, we can identify the (internal) right adjoint of this functor, $(\iota_{\tau})^{*}$, with the projection map $\pi_{\alpha}: \TwSp(A) \to \TwSp(A, \alpha)$, obtained appealing to ambidexterity and the universal property of $\TwSp(A)$ as a limit.
We now turn our attention to the commutative diagram of functors
 \[
  \begin{tikzcd}
    \TwSp(A,\tau) \arrow[r,"-\otimes Y"] \arrow[d, " (\iota_{\tau})_! "] & \TwSp(A , \tau + \sigma) \arrow[d, " (\iota_{\tau + \sigma})_! "] \\
  \TwSp(A) \arrow[r,"-\otimes_{A} Y"] & \TwSp(A)
  \end{tikzcd}
  \]
obtained by restricting the diagram from Proposition \ref{prop:total_cat_symm_mon}.
Taking right adjoints, we obtain the desired diagram.
\end{proof}

\section{The $(\infty, 2)$-Category of Twisted Spectra}

In this section, we will construct an $(\infty,2)$-category of twisted spectra, much like how May--Sigurdsson construct a bicategory of parametrized spectra~\cite{MS06}*{Chapter 17}.
We will use the $(\infty,2)$-categorical structure to definite an corresponding notion of Costenoble--Waner duality for twisted spectra.
Lastly, we return to the perspective on twisted spectra as modules over Thom spectra from Section~\ref{sec:TwSp_ModTh}, and tie the $(\infty,2)$-category of twisted spectra in with categories of bimodules and the $(\infty,2)$-categorical structure of these.

  \subsection{The $(\infty,2)$-category of twisted parametrized spectra}

Twisted spectra, over all spaces and twists, can be viewed as an $(\infty,2)$-category, along the lines of Section 2 of~\cite{Cno23}.

\begin{definition}
  The \emph{$(\infty,2)$-category of twisted spectra} $\TTwSp$ is the full subcategory of $\PPr^L_{\St}$ spanned by those stable presentable $\infty$-categories of the form $\TwSp(A,\tau)$.
\end{definition}

\begin{notation}
    Fix a twist $\tau : A \to B\Pic(\bS)$ and consider the diagonal map $\Delta : A \to A \times A$.
    We have an induced functor
    \[
    \Delta_! : \Sp^{A} \simeq \TwSp(A , -\tau + \tau) \longto \TwSp(A, -\tau \boxtimes \tau) \,.
    \]
    Evaluating this functor at the trivially parametrized sphere spectrum over $\bS$, we obtain
    \[
    \bS[\Omega A]^\tau := \Delta_! \bS \in \TwSp(A \times A , -\tau \boxtimes \tau) \,,
    \]
    which will refer to as the \emph{$\tau$-twisted spherical group ring}.
\end{notation}
Indeed, this works because of the following result.

\begin{proposition} \label{prop:FunL_exterior}
  There is an equivalence of categories
  \[
  \TwSp(A \times B, -\tau \boxtimes \sigma) \simeq \Fun^L(\TwSp(A,\tau),\TwSp(B,\sigma)) \,.
  \]
  Explicitly, it should work as follows:
  \begin{itemize}
    \item It sends $X \in \TwSp(A \times B , -\tau \boxtimes \sigma)$ to the colimit--preserving functor
    \[
    \Phi_{X} = (p_B)_!((p_A)^*(-) \otimes X) : \TwSp(A,\tau) \longto \TwSp(B,\sigma) \,,
    \]
    where $p_A : A \times B \to A$ and $p_B : A \times B \to B$ denotes the projections.
    \item It sends the left adjoint functor $F: \TwSp(A,\tau) \to \TwSp(B,\sigma)$ to the evaluation
    \[
    (1 \otimes F)(\bS[\Omega A]^{\tau}) \in \TwSp(B,\sigma)
    \]
    where we are implicitly using Proposition~\ref{prop:tensor_twisted_categories}.
  \end{itemize}
  \begin{proof}
    First, let us note the following: if $\cC$ is an invertible presentable stable $\infty$-category, that is if $\cC \in B\Pic(\bS)$, then its dual $\cC^{\vee} = \Fun^L(\cC,\Sp)$, which is also equivalent to $\cC^{\oop}$, simply corresponds to the inverse in the group structure in $B\Pic(\bS)$.
    From this, we can conclude that
    \begin{align*}
    \TwSp(A, \tau)^{\vee} &= \Fun^L(\TwSp(A,\tau),\Sp) \simeq \Fun^L(\colim_{ a \in A} \tau(a) , \Sp) \\ &\simeq \lim_{a \in A} \Fun^L(\tau(a),\Sp) \simeq \lim_{a \in A}-\tau(a)  \\ &= \TwSp(A,-\tau)
    \end{align*}
    using that $\Fun^L(-,\Sp)$ turns colimits into limits and that $\tau(a) \in B\Pic(\bS)$ for all $a \in A$.

Now we use the fact that $\TwSp(A, \tau)$ is dualizable with dual $\TwSp(A, -\tau) \simeq \Fun^{L}(\TwSp(A, \tau), \Sp )$ to conclude the general case. We have

\begin{align*}
    \TwSp(A \times B, - \tau \boxtimes \sigma) & \simeq \TwSp(A, -\tau) \otimes^{L} \TwSp(B, \sigma) \\  & \simeq
\Fun^{L}(\TwSp(A, \tau, \Sp) \otimes^{L} \TwSp(B, \sigma) \\ & \simeq \Fun^{L}(\TwSp(A, \tau), \TwSp(B ,\sigma)),
\end{align*}
where the first equivalence is Proposition \ref{prop:tensor_twisted_categories}, and the last equivalence follows by dualizability
of $\TwSp(A, \tau)$. \\

    The functor $\Phi : \TwSp(A \times B , -\tau \boxtimes \sigma) \to \Fun^L(\TwSp(A,\tau),\TwSp(B,\sigma))$ sends a twisted spectrum $X$ to its Fourier Mukai transform $\Phi_X$. The inverse is given by the composition
    \[
    \Fun^L(\TwSp(A,\tau),\TwSp(B,\sigma)) \longto \Fun^L(\TwSp(A \times A ,-\tau \boxtimes \tau),\TwSp(A \times B,-\tau \boxtimes \sigma)) \overset{\mathrm{ev}_{u_{A,\tau}}}\longto \TwSp(A \times B,-\tau \boxtimes \sigma) \,.
    \]
  \end{proof}
\end{proposition}
\begin{definition}
  Given a $(A \times B , -\tau \boxtimes \sigma)$-twisted spectrum $X$, we write
  \[
  \Phi_X = (p_B)_! \circ (- \otimes X) \circ (p_A)^* : \TwSp(A,\tau) \longto \TwSp(B,\sigma) \,.
  \]
  and call this the \emph{Fourier--Mukai transform} of $X$.
\end{definition}

\begin{rmk}
 Let us be a bit more explicit about how Proposition \ref{prop:FunL_exterior} implies the above expression. Indeed first recall we have the standard evaluation pairing
\[
\mathrm{ev}: \Fun^{L}(\TwSp(A, \tau), \TwSp(B, \sigma)) \otimes  \TwSp(A, \tau) \to \TwSp(B, \sigma)
\]
given by
\[
(X, F) \mapsto F(X) \in \TwSp(B, \sigma)
\]
The proof of Proposition \ref{prop:FunL_exterior}  shows that there is a natural equivalence of pairings

    \[
    \begin{tikzcd}[column sep = huge]
    \Fun^{L}(\TwSp(A, \tau), \TwSp(B, \sigma)) \otimes \TwSp(A, \tau) \arrow[r,"\mathrm{ev}"] \arrow[d,"\simeq"] & \TwSp(B, \sigma) \arrow[d,"\simeq"] \\
    \TwSp(A, - \tau) \otimes \TwSp(B, \sigma)  \otimes  \TwSp(A, - \tau) \arrow[r,"\mathrm{ev}"] \arrow[d,"\simeq"]  & \TwSp(B, \sigma) \arrow[d,"\simeq"] \\
    \TwSp(A \times B, - \tau \boxtimes \sigma)  \otimes \TwSp(B, \sigma)   \arrow[r,"\mathrm{ev}"]  & \TwSp(B, \sigma),
    \end{tikzcd}
    \]
where the bottom paring is given by
\[
(X, M) \mapsto (p_B)_!(p_A^*(M)\otimes X) \in   \TwSp(B, \sigma).
\]
Note that the middle pairing  above is none other than evaluation pairing $\TwSp(A, -\tau) \otimes \TwSp(A, \tau) \to \Sp $ corresponding to the duality datum on the dualizable category $\TwSp(A, -\tau)$, tensored with $\TwSp(B, \sigma)$.
\end{rmk}

The above means means that we can heuristically think of the $(\infty,2)$-category $\mathbf{TwSp}$, as shown in Table~\ref{table:infty2TwSp}.

\begin{table}[h!]
\centering
  \begin{tabular}{ | m{3.5cm} | m{6cm} | m{6cm} | }
    \hline
    & \centering $\TTwSp$ & \centering $\PPr_{\St}^L$  \tabularnewline
    \hline
    \centering objects & \centering twist $\tau : A \to B\Pic(\bS)$ & \centering presentable stable $\infty$-category \newline $\TwSp(A,\tau)$ \tabularnewline
    \hline
    \centering 1-morphism & \centering $X \in \TwSp(A \times B , -\tau \boxtimes \sigma)$ & \centering colimit preserving functor \newline $\Phi_X : \TwSp(A,\tau) \to \TwSp(B,\sigma)$ \tabularnewline
    \hline
    \centering 2-morphism & \centering map $X \to Y$ of twisted spectra & \centering natural transformation \newline $\eta : \Phi_X \to \Phi_Y$   \tabularnewline
    \hline
    \centering identity 1-morphism & \centering $\tau$-twisted spherical group ring \newline $\Delta_!^{\tau}(\bS) \in \TwSp(A \times A,-\tau \boxtimes \tau)$ & \centering identity functor \newline $1 : \TwSp(A,\tau) \to \TwSp(A,\tau)$   \tabularnewline
    \hline
    \centering vertical composition  & \centering composition of maps of twisted spectra & \centering vertical composition of natural transformations \tabularnewline
    \hline
    \centering horizontal composition of 1-morphisms &  \centering $Y \odot X \in \TwSp(A \times C , -\tau \boxtimes \rho)$ & \centering composition of functors \tabularnewline
    \hline
  \end{tabular}
  \caption{How to think of $\TTwSp$ as a full subcategory of $\PPr^L_{\St}$.}
  \label{table:infty2TwSp}
\end{table}

\begin{definition}
  Let $X \in \TwSp(A \times B, -\tau \boxtimes \sigma)$ and $Y \in \TwSp(B \times C , -\sigma \boxtimes \rho)$. The \emph{horizontal composition} of $X$ and $Y$ is defined as
  \[
  Y \odot X = (p_{A \times C})_!(p^*_{B \times C} Y \otimes p_{A \times B}^* X) \in \TwSp(A \times C, -\tau \boxtimes \rho)
  \]
  where $p_{A \times B} : A \times B \times C \to A \times B$ and $p_{B \times C} : A \times B \times C \to B \times C$ are the projections. Here, the tensor product is taken in twisted spectra over the space $A \times B \times C$.
\end{definition}

Let us prove that this actually corresponds to the horizontal composition in $\PPr^{L}_{\St}$, which is simply composition of functors.

\begin{proposition}
  Let $X \in \TwSp(A \times B , -\tau \boxtimes \sigma)$ and let $Y \in \TwSp(A \times B, -\sigma \boxtimes \rho)$. Then
  \[
  \Phi_{Y \odot X} \simeq \Phi_{Y} \circ \Phi_{X} \och \Phi_{\Delta^{\tau}_! \bS} \simeq 1_{\TwSp(A,\tau)} \,.
  \]
\end{proposition}

\begin{proof}
  The assertions essentially follows from using the projection isomorphism and Beck--Chevalley a couple of times for various maps.
  To keep track of maps let us use the notation
  \begin{equation*}
  \begin{aligned}
  p_A &: A \times B \to A \quad \quad  &q_B: B \times C \to B \quad \quad r_{A} &: A \times C \to C \\
  p_B &: A \times B \to B  \quad \quad &q_C: B\times C \to B \quad \quad  r_C &: B \times C \to C
  \end{aligned}
  \end{equation*}
  for the various projections. We will also write
  \[
  p_{A \times B} : A \times B \times C \to A \times C \quad p_{B \times C} : A \times B \times C \to B \times C \quad p_{A \times C} : A \times B \times C \to A \times C
  \]
  and
  \[
  \pi_{A} : A \times B \times C \to A \quad \pi_{B} : A \times B \times C \to B \quad \pi_{C} : A \times B \times C \to A \,.
  \]
  Per definition, we have that
  \[
  \Phi_Y \circ \Phi_X = (q_C)_! \circ (-\otimes Y) \circ (q_B)^* \circ (p_B)_! \circ (- \otimes X) \circ (p_A)^*  \,,
  \]
  and that

  \begin{align*}
  \Phi_{Y \odot X}& = (r_{C})_! \circ (- \otimes (Y \odot X)) \circ (r_A)^* \\
  &= (r_{C})_! \circ (- \otimes ((p_{A \times C})_!(p^*_{B \times C} Y \otimes p_{A \times B}^* X))) \circ (r_A)^* \,.
\end{align*}

  These are the compositions that we are trying to identify.
  We start by using Beck--Chevalley on the pullback diagram
  \[
  \begin{tikzcd}
    A \times B \times C \arrow[r,"p_{A \times B}"] \arrow[d,"p_{B \times C}"'] & A \times B \arrow[d,"p_B"] \\
    B \times C \arrow[r,"q_B"] & B
  \end{tikzcd}
  \]
  so that we have that $(q_B)^* \circ (p_B)_! \simeq (p_{B \times C})_! \circ (p_{A \times B})^*$. By combining this with the fact that the pullback $(p_{A \times B})^*$ is symmetric monoidal we obtain
  \[
  \Phi_Y \circ \Phi_X \simeq (q_C)_! \circ (-\otimes Y) \circ (p_{B \times C})_! \circ (- \otimes (p_{A \times B})^*(X)) \circ (\pi_A)^* \,.
  \]
  By using the projection isomorphism for the map $p_{B \times C}$ we can further write
  \[
  \Phi_Y \circ \Phi_X \simeq (\pi_C)_! \circ (-\otimes (p_{B \times C})^*(Y)) \circ (- \otimes (p_{A \times B})^*(X)) \circ (\pi_A)^* \,.
  \]
  Let us now look closer at the functor $\Phi_{Y \odot X}$. By using the projection isomorphism for the map $p_{A \times C}$ to make the identification
  \[
  \Phi_{Y \odot X} \simeq (\pi_{C})_! \circ (- \otimes ((p_{B \times C})^* Y \otimes (p_{A \times B})^* X)) \circ (\pi_A)^*
  \]
  which by inspection agrees with the identification we made for $\Phi_Y \circ \Phi_X$. This finishes the proof of the first assertion.
  For the second assertion, we will use the notation $p_1 , p_2 : A \times A \to A$  for the projection onto the first and second factor, respectively.
  Using the projection isomorphism for the diagonal map $\Delta : A \to A \times A$ we obtain
  \begin{align*}
    (p_2)_!\circ (- \otimes \Delta_! \bS)\circ (p_1)^* &\simeq (p_2)_! \circ \Delta_! \circ (- \otimes \bS) \circ \Delta^* (p_1)^* \\ \simeq (-)
  \end{align*}
  where we have used that $p_i \circ \Delta \simeq 1$ for $i=1,2$ and that tensoring with $\bS$ is the identity.
\end{proof}

\begin{rmk}
  Note that the above describes the composition of Fourier--Mukai functor in terms of some operation on twisted spectra. There should also be a horizontal composition on 2-morphisms, that interpreted in the language of twisted spectra should allow us to take a map $f : X \to Y$ in $\TwSp(A \times B , -\tau \boxtimes \sigma)$ and a map $g : Z \to W$ in $\TwSp(B \times C, -\sigma \boxtimes \rho)$ and construct a map
  \[
  g \odot f : Z \odot X \longto W \odot Y
  \]
  in $\TwSp(A \times C, -\tau \boxtimes \rho)$.
  Due to the funtoriality, we this should be the map
  \[
  g \odot f = (p_{A \times C})_!(p^*_{B \times C} g \otimes p_{A \times B}^* f) : Z \odot X \longto W \odot Y \,.
  \]
\end{rmk}

We will now consider the functors
\[
f_! : \TwSp(A,f^* \sigma) \longto \TwSp(B,\sigma) \och f^* : \TwSp(B,\sigma) \longto \TwSp(A,f^* \sigma)
\]
associated to a map $f : A \to B$ of spaces.
Since both of these functors are colimit-preserving, Proposition~\ref{prop:FunL_exterior} tells us there has to be $(A \times B , -f^* \sigma \boxtimes \sigma)$-twisted spectrum and $(B \times A , -\sigma \boxtimes f^* \sigma)$-twisted spectra corresponding to the functor $f_!$ and $f^*$, respectively. Let us figure out what these are.

\begin{proposition} \label{prop:FM_adjoints}
  Let $f : A \to B$ be a map of spaces and $\sigma : B \to B\Pic(\bS)$ classify a haunt over $B$. Then the following hold:
  \begin{enumerate}
    \item The pullback of $-f^* \sigma \boxtimes \sigma : A \times B \to B\Pic(\bS)$ along the map $(1,f) : A \to A \times B$ is trivial. Hence it induces a functor
    \[
    (1,f)_! : \Sp^A \longto \TwSp(A \times B , -\tau \boxtimes \sigma) \,.
    \]
    Similarly, the pullback of $-\sigma \boxtimes f^* \sigma : B \times A \to B\Pic(\bS)$ along the map $(f,1) : A \to B \times A$ is trivial and hence induces a functor
    \[
    (f,1)_! : \Sp^{A} \longto \TwSp(B \times A , -\sigma \boxtimes f^* \sigma) \,.
    \]
    \item The functor $f_! : \TwSp(A,f^* \sigma) \to \TwSp(B,\sigma)$ corresponds to the $(A\times B,-\tau \boxtimes \sigma)$-twisted spectrum $(1,f)_! \bS$ under the Fourier--Mukai transform. That is:
    \[
    \Phi_{(1,f)_! \bS} \simeq f_! \,.
    \]
    \item The functor $f^* : \TwSp(B,\sigma) \to \TwSp(A,f^* \sigma)$ corresponds to the $(B \times A , -\sigma \boxtimes f^* \sigma)$-twisted spectrum $(f,1)_! \bS$. That is:
    \[
    \Phi_{(f,1)_! \bS } \simeq f^* \,.
    \]
  \end{enumerate}
  \begin{proof}
    The first part of the theorem is straight-forward. The second and third statement are proved in the same way, so we only prove the second statement, leaving the third for the reader. By using the projection isomorphism for the map $(1,f)$, we obtain
    \begin{align*}
    \Phi_{(1,f)_! \bS} &= (p_B)_! \circ (- \otimes (1,f)_! \bS)\circ (p_A)^* \\
    & \simeq (p_B)_! \circ (1,f)_! \circ ( - \otimes \bS) \circ (1,f)^* \circ (p_A)^*
    \end{align*}
    from which the result follows by observing that $p_B \circ (1,f)$ is just $f$ and that $p_A \circ (1,f)$ is the identity, and that $\bS$ is the unit for the symmetric monoidal structure.
  \end{proof}
\end{proposition}

\subsection{Dualizability of Twisted Spectra} \label{sec:dualizability_twisted_spectra}

Note that, per construction, the Fourier--Mukai transform of a twisted spectrum $X \in \TwSp(A \times B,-\tau \boxtimes \sigma)$ is colimit-preserving, so it must admit a right adjoint, namely the composition
\[
\TwSp(B,\sigma) \overset{(p_B)^*}\longto \TwSp(A \times B , (p_B)^* \sigma) \overset{\hom(X,-)}\longto \TwSp(A \times B , (p_B)^* \sigma + \tau \boxtimes \sigma) \overset{(p_A)_*}\longto \TwSp(A,\tau) \,.
\]
Dualizability of a twisted spectrum is going to be phrased in terms of whether this functor is itself colimit-preserving.

\begin{definition} \label{def:dualizable}
  We say that the $(A \times B , -\tau \boxtimes \sigma)$-twisted spectrum $X$ is \emph{dualizable} if the functor $\Phi_X : \TwSp(A,\tau) \to \TwSp(B,\sigma)$ is an internal left adjoint in $\Pr_{\St}^L$. In other words, it is dualizable if the right adjoint to its Fourier--Mukai transform is itself colimit-preserving.
\end{definition}

In particular, notice that if $X$ is dualizable, then by Proposition~\ref{prop:FunL_exterior} the right adjoint must be on the form $\Phi_{X^{\vee}}$ for some essentially unique $(B \times A , - \sigma \boxtimes \tau)$-twisted spectrum $X^\vee$. We refer to this $X^{\vee}$ as the \emph{dual} of $X$.

\begin{example}
  Since the right adjoint of $f_!$ is $f^*$, which itself admits a left adjoint, we can use Proposition~\ref{prop:FM_adjoints} to conclude that the dual of the $\TwSp(A \times B , -f^* \sigma \boxtimes \sigma)$-twisted spectrum $X = (1,f)_! \bS$ is the $\TwSp(B \times A, -\sigma \boxtimes f^* \sigma)$-twisted spectrum $X^\vee = (f,1)_! \bS$.
\end{example}

We will now recast some of the above discussion in the extreme cases that one of the spaces $A$ or $B$ is a point.
We will abusively denote the projections $A \to \ast$ and $B \to \ast$ simply by $A$ and $B$, respectively.
We first note the following:
\begin{itemize}
  \item If $A$ is a point, then the $(B,\sigma)$-twisted spectrum $X$ corresponds to the colimit-preserving functor
  \[
  \Phi_X : \Sp \longto \TwSp(B,\sigma) \,, \quad Y \mapsto B^*(Y) \otimes X \,.
  \]
  \item If $B$ is a point, then the $(A , -\tau)$-twisted spectrum $X$ corresponds to the colimit-preserving functor
  \[
  \Phi_X : \TwSp(A,\tau) \longto \Sp \,, \quad Y \mapsto A_! (Y \otimes X) \,.
  \]
\end{itemize}

Per the above, a twisted spectrum can thus be viewed as a colimit-preserving functor in essentially two ways: as a functor with either target or source $\Sp$.
 This gives rise to a lot of interesting structures.
 To be able to be slightly more explicit in our writing, let us set the following notation in this section.

\begin{notation}
  Let $X \in \TwSp(A,\tau)$. We retain the notation
  \[
  \Phi_X : \Sp \to \TwSp(A,\tau) \,, \quad Y \mapsto A^*(Y) \otimes X
  \]
  for the Fourier--Mukai transform thought of as a functor $\Sp \to \TwSp(A,\tau)$, but write
  \[
  \Psi_X : \TwSp(A,-\tau) \to \Sp \,, \quad Y \mapsto A_!(Y \otimes X)
  \]
  for the Fourier--Mukai transform thought of as a functor $\TwSp(A,-\tau) \to \Sp$.
\end{notation}

In particular, note that the $\Phi$- and $\Psi$-notation are exchanged under duality: if we think of $X \in \TwSp(A,\tau)$ as the functor $\Phi_X$, then we must think of its dual $X^{\vee} \in \TwSp(A,-\tau)$ as the functor $\Psi_{X^{\vee}}$, and vice versa.
The following preliminary lemma might be useful.

\begin{lemma}
  Let $X \in \TwSp(A,-\tau)$ and let $Y \in \TwSp(A,\tau)$. Then
  \[
  \Psi_X \circ \Phi_{Y} \simeq \Phi_{A_!(X \otimes Y)} \och \Phi_Y \circ \Psi_{X} \simeq \Phi_{X \boxtimes Y}
  \]
  where the right hand sides refer to the Fourier--Mukai transforms associated to $A_!(Y \otimes X) \in \Sp$ and $X \boxtimes Y \in \TwSp(A \times A, -\tau \boxtimes \tau)$, respectively.
  \begin{proof}
    The first statement follows from using the projection isomorphism for the map $A \to \ast$ to obtain
    \begin{align*}
      \Psi_X \circ \Phi_Y &\simeq A_! \circ (- \otimes Y) \circ (- \otimes X) \circ A^* \\ &\simeq (- \otimes A_!(X \otimes Y)) \circ A^* \\ &\simeq \Phi_{A_!(X \otimes Y)}
    \end{align*}
    For the second statement, let us try to try to understand $\Phi_{X \boxtimes Y}$.
    We will use the notation $p_1 , p_2 : A \times A \to A$ refers to the first and second projection, respectively. Then we have that
    \begin{align*}
    \Phi_{X \boxtimes Y} &= (p_2)_! \circ (- \otimes (X \boxtimes Y)) \circ (p_1)^* \\ &\simeq (p_2)_! \circ (- \otimes ((p_1)^* X \otimes (p_2)^* Y)) \circ (p_1)^*
    \end{align*}
    by using Lemma~\ref{lem:boxtimes_proj_pullbacks}. By using that $(p_1)^*$ is symmetric monoidal and the projection isomorphism for the map $p_2$ we obtain
    \begin{align*}
    \Phi_{X \boxtimes Y} &\simeq (- \otimes Y)\circ (p_2)_!\circ (p_1)^* \circ  (- \otimes X) \,.
    \end{align*}
    We finish the proof by using Beck--Chevalley for the pullback diagram
    \[
    \begin{tikzcd}
      A \times A \arrow[d,"p_2"'] \arrow[r,"p_1"] & A \arrow[d] \\
      A \arrow[r] & \ast \,.
    \end{tikzcd}
    \]
    so that we can write
    \[
    \Phi_{X \boxtimes Y} \simeq (- \otimes Y) \circ A^* \circ A_! \circ (- \otimes X) = \Phi_Y \circ \Psi_X \,.
    \]
  \end{proof}
\end{lemma}

The reason for the auxiliary notation $\Psi_X$ is that duality of $X$ in the sense of the functor $\Psi_X$ being an internal left adjoint is not super interesting from a structural point of view; it just reduces to Spanier--Whitehead dualizablity.

\begin{proposition}
  Consider a twisted spectrum $X \in \TwSp(A,-\tau)$, which we view as a colimit-preserving functor
  \[
  \Psi_X : \TwSp(A,\tau) \longto \Sp \,.
  \]
  The dualizability of $X$ in the sense of Definition~\ref{def:dualizable} then reduces to dualizability of $X$ in the symmetric monoidal $\infty$-category $\TwSp(A)$.
\end{proposition}

\begin{proof}
  Let us write $D X$ for the supposed dual of $X$.
  Dualizability in the sense of Definition~\ref{def:dualizable}, then just says that $\Phi_{DX} : \Sp \to \TwSp(A,\tau)$ is right adjoint to $\Psi_X$.
  In particular, this means that we have a unit and counit
  \[
  \eta : 1 \to \Phi_{DX} \circ \Psi_{X} \och \epsilon : \Psi_X \circ \Phi_{DX} \to 1 \,,
  \]
  which are natural transformations of functors $\TwSp(A,\tau) \to \TwSp(A,\tau)$ and $\Sp \to \Sp$ respectively, satisfying the triangle equalities. We recall that these triangle equality simply says that the compositions
  \[
  \Psi_X \overset{\Psi_X \eta}\to \Psi_X \Phi_{DX} \Psi_X \overset{\epsilon \Psi_{X}}\to \Psi_X \och \Phi_{DX} \overset{\eta \Phi_{DX}}\to \Phi_{DX} \Psi_X \Phi_{DX} \overset{\Phi_{DX} \epsilon}\to \Phi_{DX} \,,
  \]
  are the identity natural transformations on $\Psi_X$ and $\Phi_{DX}$, respectively.
  In terms of $(A \times A , -\tau \boxtimes \tau)$-twisted spectra and ordinary spectra, the unit and counit just corresponds to maps
  \[
  \eta : \Delta_{!}^\tau \bS \to DX \boxtimes X \och \epsilon : A_!(X \otimes DX) \to \bS
  \]
  which under the $(\Delta_! \vdash \Delta^{*})$-adjunction and the $(A_! \vdash A^*)$-adjunction corresponds to maps
  \[
  \eta : \bS \to DX \otimes X \och \epsilon : X \otimes DX \to \bS \,,
  \]
  of parametrized spectra over $A$.
  Under these identifications, the triangle equalities translate to the condition that the  compositions
  \[
  X \overset{1 \otimes \eta}\to X \otimes DX \otimes X \overset{\epsilon \otimes 1}\to X \och DX \overset{\eta \otimes 1}\to DX \otimes X \otimes DX \overset{1 \otimes \epsilon}\to DX
  \]
  are the identity maps on $X$ and $DX$, respectively. This is precisely the condition for dualizability in the symmetric monoidal category $\TwSp(A)$.
\end{proof}

The other kind of dualizability does give us something new, though.

\begin{definition}
  Suppose that $A$ is a point and consider a twisted spectrum $X \in \TwSp(B,\sigma)$, which we view as a colimit-preserving functor
  \[
  \Phi_X : \Sp \to \TwSp(B,\sigma)
  \]
  via the Fourier--Mukai transform.
  The dualizablity of $X$ in the sense of Definition~\ref{def:dualizable} will then be referred as \emph{Costenoble--Waner dualizability} and the dual $D^{CW} X \in \TwSp(B,-\sigma)$ will be referred as the \emph{Costenoble--Waner dual}.
\end{definition}

\begin{example}
  Suppose that $A$ is connected and consider the inclusion of a point $a \in A$. By abuse of notation, we will also denote the inclusion $a$. For any twist $\tau : A \to B\Pic(\bS)$ we have an induced functor $a^{\tau}_! : \Sp \longto \TwSp(A,\tau)$, where we have decorated the superscript of the functor to remember what twist we are dealing with. Consider the object $X = a^{\tau}_! \bS$. We claim that this is Costenoble--Waner dual to $a_!^{-\tau} \bS$. Indeed, one can easily check that
  \[
  \Phi_{a^{\tau}_!(\bS)} \simeq a^{\tau}_! : \Sp \longto \TwSp(A,\tau)
  \]
  by using the projection isomorphism. The right adjoint to $\Phi_{a^{\tau}_! \bS}$ is hence the functor $a^{*,\tau} : \TwSp(A,\tau) \to \Sp$.
\end{example}

\subsection{Bimodule perspective on the $(\infty,2)$-category of twisted spectra}

Let us restrict ourselves to spaces that are connected. We have already seen that we have natural equivalences
\[
\TwSp(A,\tau) \simeq \Mod_{\Th(\Omega A)}
\]
of $\infty$-categories whenever $A$ is pointed and connected.
In this section, we will identify the subcategory of the $(\infty,2)$-category of twisted spectra with a subcategory of the Morita $(\infty,2)$-category.
That is, we will recast some of the discussion on $(\infty,2)$-categories in terms of bimodules of Thom spectra.
We start with the following definition.

\begin{definition}
  Let $\mathbf{ThMod}$ denote the full subcategory of $\mathbf{Pr}^L_{\St}$ spanned by those stable presentable $\infty$-categories of the form $\Mod_{\Th(\alpha)}$ where $\alpha : G \to \Pic(\bS)$ is some object in $\Grp_{/\Pic(\bS)}$.
\end{definition}

This definition makes sense since for every pair of $\bE_1$-rings $R$ and $S$, we have an equivalence
\[
{}_{R} \mathrm{BiMod}_{S} \simeq \Fun^L(\Mod_R , \Mod_S)
\]
of $\infty$-categories; see for example~\cite{HA}*{Section 4.8}. In particular, the colimit-preserving functor associated to the $R-S$-bimodule $M$ is
\[
F_M = - \otimes_R M : \Mod_R \longto \Mod_S \,.
\]
It is well-known that composition of such functors corresponds to tensoring bimodules over their joint action: if $M$ is an $R-S$-bimodule and $N$ is an $S-T$-bimodule, then the composition $F_N \circ F_M$ corresponds to the $R-T$-bimodule $M \otimes_S N$.
This means that we can heuristically think of the $(\infty,2)$-category $\mathbf{ThMod}$  as shown in Table~\ref{table:ThMod}.

\begin{table}[h!]
\centering
  \begin{tabular}{ | m{3.5cm} | m{6cm} | m{6cm} | }
    \hline
    & \centering $\mathbf{ThMod}$ & \centering $\PPr_{\St}^L$ \tabularnewline
    \hline
    \centering objects & \centering $\bE_1$-algebra $\Th(\alpha)$ & \centering presentable stable $\infty$-category \newline $\Mod_{\Th(\alpha)}$   \tabularnewline
    \hline
    \centering 1-morphism  & \centering $\Th(\alpha)-\Th(\beta)$-bimodule $M$ & \centering colimit preserving functor \newline $F_M : \Mod_{\Th(\alpha)} \to \Mod_{\Th(\beta)}$ \tabularnewline
    \hline
    \centering 2-morphism & \centering bimodule map $M \to N$ & \centering natural transformation \newline $\eta : F_M \to F_N$  \tabularnewline
    \hline
    \centering identity 1-morphism & \centering diagonal bimodule $\Th(\alpha)$ & \centering identity functor $1 : \Mod_{\Th(\alpha)} \to \Mod_{\Th(\alpha)}$ \tabularnewline
    \hline
    \centering vertical composition & \centering vertical composition of \newline natural transformations & \centering composition of bimodule maps \tabularnewline
    \hline
    \centering horizontal composition on 1-morphism & \centering composition of functors & \centering tensoring over the joint common action \tabularnewline
    \hline
  \end{tabular}
  \caption{How to think of $\mathbf{ThMod}$ as a full subcategory of $\PPr^L_{\St}$. }
  \label{table:ThMod}
\end{table}

Essentially, we would like to know that, under the equivalence of twisted spectra and modules over Thom spectra, that the Fourier--Mukai transform precisely corresponds to the functor associated to a bimodule.
This is the content of the following result.

\begin{proposition}
  Let $\alpha : G \to \Pic(\bS)$ and $\beta : H \to \Pic(\bS)$ and suppose that we are given a $\Th(\alpha) - \Th(\beta)$-module $M$. Suppose that  $X$ is the twisted spectra corresponding to $M$ under the equivalence
  \[
  \TwSp(BG \times BH , -B\alpha \boxtimes B\beta) \simeq {}_{\Th(\alpha)} \mathrm{BiMod}_{\Th(\beta)} \,.
  \]
  Then the diagram
  \[
  \begin{tikzcd}
  \TwSp(BG, B\alpha) \arrow[r,"\Phi_X"] \arrow[d,"\simeq"] & \TwSp(BH,B\beta) \arrow[d,"\simeq"] \\
  \Mod_{\Th(\alpha)}  \arrow[r,"F_M"] & \Mod_{\Th(\beta)}
\end{tikzcd}
  \]
  commutes.
  \begin{proof}
    Let us study the adjunctions
    \[
    (- \otimes \TwSp(BG,B\alpha)) \vdash \Fun^L(\TwSp(BG,B\alpha),-) \och (- \otimes \Mod_{\Th(\alpha)} \vdash \Fun^L(\Mod_{\Th(\alpha)},-))
    \]
    of functors $\Pr^{L}_{\St} \to \Pr^{L}_{\St}$.
    Due to Theorem~\ref{thm:TwSp_Mod_natural}, we know that these adjunctions are equivalent to one another, so for ease of notation let us not make a distinction between them, and simply write $- \otimes \cC \vdash \Fun^L(\cC,-)$.
    Consider the counit of the adjunction, the evaluation
    \[
    \mathrm{ev}_{\mathcal{D}} :\Fun^L(\cC,\calD) \otimes  \cC \to \calD \,.
    \]
    Since this is a natural transformation we know that the diagram
    \[
    \begin{tikzcd}[column sep = huge]
    \Fun^L(\cC,\TwSp(BH,B\beta)) \otimes \cC \arrow[r,"\mathrm{ev}_{\TwSp(BH,B\beta)}"] \arrow[d,"\simeq"] & \TwSp(BH,B\beta) \arrow[d,"\simeq"] \\
    \Fun^L(\cC,\Mod_{\Th(\beta)}) \otimes \cC  \arrow[r,"\mathrm{ev}_{\Mod_{\Th(\beta)}}"] & \Mod_{\Th(\beta)}
    \end{tikzcd}
    \]
    commutes.
    The result follows from noting that the functors $\Phi_X$ and $F_M$ are precisely $ev_{\TwSp(BH,B\beta)}(X \otimes -)$ and $ev_{\Mod_{\Th(\beta)}}(M \otimes -)$ under the equivalences
    \[
    \Fun^{L}(\TwSp(BG,B\alpha),\TwSp(BH,B\beta)) \simeq \TwSp(BG \times BH , -B\alpha \boxtimes B \beta) \,,
    \]
    and
    \[
    \Fun^L(\Mod_{\Th(\alpha)},\Mod_{\Th(\beta)}) \simeq {}_{\Th(\alpha)} \mathrm{BiMod}_{\Th(\beta)} \,,
    \]
    respectively.
  \end{proof}
\end{proposition}

\noindent Putting this all together, we deduce the following statement:
\begin{theorem}
  Let $\mathbf{TwSp}^{\mathrm{conn}}$ denote the full subcategory of $\PPr_{\St}^L$ spanned by those stable $\infty$-categories of the form $\TwSp(BG,B\alpha)$ for some object $\alpha \in \Grp_{/\Pic(\bS)}$.
  Then there is an equivalence
  \[
  \TTwSp^{\mathrm{conn}} \simeq \mathbf{ThMod}
  \]
  of $(\infty,2)$-categories.
\end{theorem}

\section{Derived Algebraic Geometry Perspective} \label{derivedstory}

The notion of a twisted spectrum on a topological space $B$ has a natural interpretation within the setting of spectral algebraic geometry.
Roughly, this is a setting of (derived) algebraic geometry where one takes the basic affine objects to be connective $\bE_\infty$-ring spectra.
Gluing these together in the appropriate homotopical context, one obtains spectral schemes, and more generally, spectral stacks.
The goal of this section is to identify a twist on a space $B$ with a \emph{Brauer class}, in the sense of derived algebraic geometry, on the associated constant stack $B$.

\subsection{Derived stacks}

Let $\CAlg^{\on{cn}}$ denote the $\infty$-category of connective $\bE_\infty$-rings. One endows its opposite category with the $\infty$-categorical variant of the \'{e}tale topology, cf. ~\cite[Chapter 2.4]{HAG}. Our treatment here will follow the functor of points approach to defining stacks.

\begin{definition}
The $\infty$-category of \emph{spectral stacks}  is defined to be the full subcategory $\on{dStk}:= \Fun^{\tau}(\CAlg^{\on{cn}}, \cS) \subset \Fun(\CAlg^{\on{cn}}, \cS)$  of space valued functors which satisfy descent with respect to the \'{e}tale topology.
\end{definition}

\begin{rmk}
The definition of a derived stack is manifestly very general; one typically focuses on stacks which have certain properties, for example spectral schemes or spectral Deligne-Mumford stacks.
\end{rmk}

Since $\on{dStk}$ is a geometric localization of a presheaf category, it naturally forms an $\infty$-topos. As such, there is a geometric morphism adjunction
$$
c^* :\cS \leftrightarrows \on{dStk} : \Gamma
$$
where $c^*$ denotes the \emph{constant}, or \emph{Betti} stack, and $\Gamma$ denotes the global sections morphism.

\begin{example}
    For the purposes of this work, an important class of examples of spectral stacks will be the constant stacks, i.e. the essential image of the functor $c^*(-)$ above.
\end{example}

As in classical algebraic geometry, one probes the geometry of a spectral stack by studying its category of quasi-coherent sheaves. Given a fixed spectral stack $F$, the symmetric monoidal $\infty$-category of quasi-coherent sheaves as a limit
$$
\QCoh(F) \simeq \lim_{\on{Spec} A \to F} \Mod(A)
$$
of module categories. If $F$ is a constant stack associated to a topological space one has the following identification.

\begin{proposition}
    There is an equivalence
    $$
    \QCoh(c^*(B))\simeq \Fun(B, \Sp)
    $$
between quasi-coherent sheaves on $c^*(B)$ and parametrized categories on $B$.
\end{proposition}

\subsection{The Brauer stack and derived Azumaya algebras}

As we saw in Section 1, the Brauer space of a $\bE_\infty$ ring $R$ captures the invertible objects with respect to the Lurie tensor product on compactly generated $R$-linear categories.
Such objects in this affine case correspond precisely to derived Azumaya algebras over $R$, whose definition we now recall.

\begin{definition}
   An $R$-algebra $A$ is an \emph{Azumaya $R$-algebra} if $A$ is a compact generator of $\Mod_R$ and if the natural map
   $$
   A \otimes_R A^{\on{op}} \to \End_{R}(A)
   $$
   is an equivalence of $R$-algebras.
\end{definition}

The assignment $R \mapsto \Br(R)$ satisfies \'{e}tale descent and hence can be extended to the $\infty$-category of spectral stacks.
In fact, there is a spectral stack $\mathbf{Br}$ \cite{AG14}*{Section 6.1} such that for an arbitrary $F \in \on{dStk}$, one has an equivalence
$$
\Map_{\on{dStk}}(F, \mathbf{Br}) \simeq \Br(F) \,.
$$
The notion of an Azumaya algebra globalizes, as well. Namely, one obtains a spectral stack $\mathbf{Az}$ whose value on a connective $\bE_\infty$ ring $R$ is exactly the space of Azumaya algebras. A derived Azumaya algebra on a stack $F$ will then just be a map $F \to \mathbf{Az}$. There is a natural map
$$
\mathbf{Az} \to \mathbf{Br}
$$
of functors which on $R$-points is simply sending sending $A \mapsto \Mod_A$.
A motivating question, going back to Grothendieck's study of Azumaya algebras and \'{e}tale cohomology, is whether or not every Brauer class of a stack $F$ arises from a derived Azumaya algebra over $F$.
We may phrase this in this language as whether or not the map $F \to \mathbf{Br}$ classifying a Brauer class lifts across the map $\mathbf{Az} \to \mathbf{Br}$ to a map $F \to \mathbf{Az}$.

Now, by a result of To\"{e}n in the simplicial case (cf. ~\cite[Theorem 3.7]{derivedaz}) and Antieau-Gepner in the spectral setting at hand (cf. ~\cite[Theorem 6.1]{AG14}), for a qcqs spectral scheme the above is true. We remark that this was extended by Lurie  in ~\cite{SAG} to the case of spectral algebraic spaces using his theory of scallop decompositions.

\subsection{Betti Stacks and Twisted Spectra}
We now tie the ensuing discussion with our study of twisted spectra. Let $B$ be a fixed topological space whose associated Betti stack we denote by  $c^*(B)$

\begin{proposition}
    There exists an equivalence
    \[
    \Map_{\cS}(B, \Br(\bS)) \simeq \mathbf{Br}(c^*(B)) \,.
    \]
\end{proposition}

\begin{proof}
    This is simply a consequence of the constant stack/global sections adjunction described above. In particular
    $$
    \mathbf{Br}(c^*(B)) \simeq \Map_{\on{dStk}}(c^*(F), \mathbf{Br}) \simeq \Map_{\on{dStk}}(F, \Gamma(\mathbf{Br})).
    $$
We now identify $\Gamma(\mathbf{Br})$, the global sections of the stack $\mathbf{Br}$. This is computed by its value on the final object, $\on{Spec} \bS$; thus  $\Gamma(\mathbf{Br}) \simeq \Br(\bS)$.
\end{proof}

Thus we may identify a twist on a topological space $B$ with a Brauer class on the associated Betti stack. We now rephrase the realizability question of Grothendieck in the following relevant form.

\begin{question}
Let $c^*(B)$ be the Betti stack associated to a topological space, and let $\alpha : B \to B \Pic(\bS)$ be a twist corresponding to a Brauer class on $c^*(B)$. Does there exist an algebra object $\cA$ in $\on{QCoh}(c^*(B))$ for which $\TwSp(B, \alpha) \simeq \Mod_{\cA}$ as $\on{QCoh}(c^*(B))$-linear categories?
\end{question}

In fact this question is easily seen to have a positive answer when the space $B$ is pointed connected. Recall from Section \ref{sec:TwSp_ModTh} that there is a natural equivalence
\[
\TwSp(B, \alpha) \simeq \on{RMod}_{\Th(\Omega \alpha)}
\]
Via the equivalence $\Sp^{B} \simeq \on{RMod}_{\bS[\Omega B]}$, (eg by ~\cite[Lemma 4.46]{CCRY23}) where the symmetric monoidal structure on the right hand side is transported over from the pointwise monoidal structure on $\Sp^{B}$, we may  view the $E_1$-algebra  $\Th(\Omega \alpha)$ as corresponding to an $E_1$-algebra object in $ \Sp^{B}$. In particular, this lies in the essential image of the functor
\[
\Theta: \on{Alg}_{\Sp^{B}} \to \Mod_{\Sp^{B}}(\on{Pr}^{L}_{\St}),
\]
sending
$A \mapsto \on{RMod}_A$, cf. ~\cite[Remark 4.8.5.17]{HA}.
Thus, we may conclude that Brauer classes on constant stacks arise from algebra objects over constant stacks  $c^*B$ for $B$ pointed  and connected. Hence we may think of $\Th(\Omega \alpha)$ as a type of (derived) Azumaya algebra over $c^*B$.

\end{document}